\documentclass[10pt]{amsart}
\usepackage{amsfonts}
\usepackage{mathrsfs}
\usepackage{amsmath,amssymb,latexsym,esint,mathrsfs}
\usepackage{verbatim,cite}
\usepackage[left=2.8cm,right=2.8cm,top=2.9cm,bottom=2.9cm]{geometry}

\usepackage{color,enumitem,graphicx}
\usepackage[colorlinks=true,urlcolor=blue, citecolor=red,linkcolor=blue,
linktocpage,pdfpagelabels, bookmarksnumbered,bookmarksopen]{hyperref}

\usepackage[hyperpageref]{backref}
\usepackage[english]{babel}

\newtheorem{theorem}{Theorem}[section]
\newtheorem{lemma}[theorem]{Lemma}
\newtheorem{corollary}[theorem]{Corollary}
\newtheorem{proposition}[theorem]{Proposition}

\newtheorem{remark}[theorem]{Remark}
\newtheorem{definition}[theorem]{Definition}
\numberwithin{equation}{section}
\makeindex

\newcommand{\s}{\section}

\newcommand{\R}{\mathbb R}

\renewcommand{\div}{{\rm div}}

\newcommand{\bt}{\begin{theorem}}
\newcommand{\et}{\end{theorem}}
\newcommand{\bl}{\begin{lemma}}
\newcommand{\el}{\end{lemma}}
\newcommand{\bd}{\begin{definition}}
\newcommand{\ed}{\end{definition}}
\newcommand{\bc}{\begin{corollary}}
\newcommand{\ec}{\end{corollary}}
\newcommand{\bp}{\begin{proof}}
\newcommand{\ep}{\end{proof}}
\newcommand{\bx}{\begin{example}}
\newcommand{\ex}{\end{example}}
\newcommand{\bi}{\begin{exercise}}
\newcommand{\ei}{\end{exercise}}
\newcommand{\bo}{\begin{prop}}
\newcommand{\eo}{\end{prop}}
\newcommand{\br}{\begin{remark}}
\newcommand{\er}{\end{remark}}
\newcommand{\be}{\begin{equation}}
\newcommand{\ee}{\end{equation}}
\newcommand{\ba}{\begin{align}}
\newcommand{\ea}{\end{align}}
\newcommand{\bn}{\begin{enumerate}}
\newcommand{\en}{\end{enumerate}}
\newcommand{\bg}{\begin{align*}}
\newcommand{\bcs}{\begin{cases}}
\newcommand{\ecs}{\end{cases}}

\newcommand{\bean}{\begin{eqnarray*}}
\newcommand{\eean}{\end{eqnarray*}}

\renewcommand{\S}{{\mathcal S}}
\renewcommand{\O}{{\mathcal O}}
\newcommand{\T}{{\mathcal T}}

\begin{document}

\title[Nehari manifold for critical fractional systems]{The Nehari manifold for fractional systems \\ involving critical nonlinearities}

\author[X. He]{Xiaoming He}
\author[M.\ Squassina]{Marco Squassina}
\author[W.\ Zou]{Wenming Zou}

\address[X. He]{College of Science,
    Minzu University of China
    \newline \indent
    Beijing 100081, China}
\email{xmhe923@muc.edu.cn}

\address[M. Squassina]{Dipartimento di Informatica,
    Universit\`a degli Studi di Verona
    \newline\indent
    Strada Le Grazie 15, I-37134 Verona, Italy}
\email{marco.squassina@univr.it}

\address[W.\ Zou]{Department of Mathematical Sciences,
     Tsinghua University
        \newline \indent
      Beijing 100084, China}
\email{wzou@math.tsinghua.edu.cn}


\thanks{X.\ He is  supported by NSFC
    (11371212, 10601063, 11271386) while W.\ Zou
    by NSFC (11371212, 11271386)}


\begin{abstract}
We study the combined effect of
concave and convex nonlinearities on the number of positive
solutions for a fractional system
involving critical Sobolev exponents. With the help of the
Nehari manifold, we prove that the system admits at least two
positive solutions when the pair of parameters $(\lambda,\mu)$
belongs to a suitable subset of $\R^2$.
\end{abstract}

\maketitle

\begin{center}
\begin{minipage}{9.5cm}
\small
\tableofcontents
\end{minipage}
\end{center}

\medskip

  \s{Introduction}
This paper is concerned with the multiplicity of positive
solutions for the following elliptic system involving the fractional
Laplacian
\be\label{zwm=1}
\begin{cases}
            (-\Delta)^su=\lambda|u|^{q-2}u+\frac{2\alpha}{\alpha+\beta}|u|^{\alpha-2}u|v|^{\beta} &\mbox{in}~ \Omega,\vspace{0.1cm}\\
           (-\Delta)^sv=\mu|v|^{q-2}v+\frac{2\beta}{\alpha+\beta}|u|^{\alpha}|v|^{\beta-2}v &\mbox{in}~ \Omega,\vspace{0.1cm}\\
            u=v=0               &\mbox{on} ~\partial \Omega,
 \end{cases}\ee
where  $\Omega\subset\R^N$  is a smooth  bounded domain, $\lambda,\mu>0$, $1<q<2$ and
    $\alpha>1,\beta>1$ satisfy $\alpha+\beta=2^*_s=2N/(N-2s)$, $s\in(0,1)$ and $N>2s$.
    When $\alpha=\beta, \alpha+\beta=p\leq 2^*_s, \lambda=\mu$ and $u=v$, problem
\eqref{zwm=1} reduces to the semilinear scalar fractional
elliptic equation
 \be\label{zwm=2}
\begin{cases}
            (-\Delta)^su=\lambda|u|^{q-2}u+|u|^{p-2}u &\mbox{in}~ \Omega, \vspace{0.1cm}\\
                      u=0               &\mbox{on} ~\partial
                      \Omega.
     \end{cases}\ee
Recently, a great attention has been focused on the study of nonlinear
problems like \eqref{zwm=2} which involve the fractional Laplacian.\ This type of operators
naturally arises in physical situations such as thin obstacle problems, optimization,
 population dynamics,  geophysical fluid dynamics,
mathematical finance, phases transitions, straitified materials,
anomalous diffusion, crystal dislocation, soft thin films,
semipermeable membranes, flames propagation, conservation laws,
ultra-relativistic limits of quantum mechanics, quasi-geostrophic
flows, multiple scattering, minimal surfaces, materials science and
water waves, see \cite{NPV}.  We refer to \cite{CLB,SV1,
SV7,Seira, Wei,Yu} for the subcritical case and to
\cite{B1,B2,CD,CDU,CT1,SZ,SV6,Tan} for the critical case.
In the remarkable paper \cite{CS}, Caffarelli and Silvestre gave a new formulation
of the fractional Laplacian through Dirichlet-Neumann maps. This is
extensively used in the recent literature since it allows to
transform nonlocal problems to local ones, which permits to use
variational methods.\ For example, Barrios, Colorado, de Pablo and S\'{a}nchez \cite{B1} used the idea of
the $s$-harmonic extension and studied the effect of lower order perturbations in
 the existence of positive solutions of \eqref{zwm=2}.\  Br\"{a}ndle, Colorado and  de Pablo   \cite{BCP} investigated the fractional
elliptic equation \eqref{zwm=2} involving concave-convex nonlinearity, and
obtained an analogue multiplicity result  to  the problem considered by
Ambrosetti, Br\'ezis and Cerami in  \cite{Ambrosetti}. In the case $q=2$ and $p=2^*_s,$
Servadei and Valdinoci \cite{SV6} studied \eqref{zwm=2}
and extended the classical Br\'ezis-Nirenberg result \cite{BN}
to the nonlocal case.
In \cite{CT}, Cabr\'{e} and Tan defined $(-\Delta)^{1/2}$ through the spectral decomposition of the Laplacian
operator on $\Omega$ with zero Dirichlet boundary conditions. With
classical local techniques, they established existence of positive
solutions for problems with subcritical nonlinearities, regularity
and $L^{\infty}$-estimates for weak solutions. In particular, Tan \cite{Tan} considered
 \be\label{zwm=3}
\begin{cases}
            (-\Delta)^{1/2}u=\lambda u+u^{\frac{N+1}{N-1}} &\mbox{in}~ \Omega,\vspace{0.1cm}\\
                      u=0               &\mbox{on} ~\partial \Omega,
\end{cases}\ee
investigating the solvability (see also \cite{Yu} for a subcritical situation).
Very recently,  Colorado,  de Pablo, and
S\'{a}nchez  \cite{CDU}  studied the following nonhomogeneous fractional equation involving critical Sobolev exponent
 \begin{equation*}
\begin{cases}
            (-\Delta)^su=|u|^{2^*_s-2}u+f(x) &\mbox{in}~ \Omega,\vspace{0.1cm}\\
                      u=0               &\mbox{on} ~\partial
                      \Omega.
\end{cases}
\end{equation*}
and proved existence and multiplicity of solutions
under appropriate conditions on the size of $f$.  For the same problems,
Shang,  Zhang and Yang \cite{SZ}  obtained   similar results.
\vskip3pt
\noindent
 The analogue problems to  \eqref{zwm=1} for the Laplacian operator have been studied extensively in recent years,
see \cite{Alves,CM,Han,HL,Wu1,Wu2}  and the
 references therein.\ In particular,  Hu and  Lin  \cite{HL}  studied the Laplacian system  with critical growth and obtained
  the existence and multiplicity results of positive solutions by variational methods.
 \vskip3pt
 \noindent
The purpose of this paper is to study system \eqref{zwm=1}
in the critical case $\alpha+\beta=2^*_s$. Using variational methods and a {\em Nehari manifold decomposition},
       we  prove that system \eqref{zwm=1} admits at least two positive solutions when the pair of parameters $(\lambda,\mu)$
   belongs to a certain subset of  $\R^2$. To our best knowledge, there are just a few results in the literature on the fractional
system \eqref{zwm=1} with both concave-convex nonlinearities and
critical growth terms. We point out that we adopt in the paper the
{\em spectral} (or {\em regional}) definition of the fractional
laplacian in a bounded domain based upon a Caffarelli-Silvestre type
extension (see \cite{CDU}), and not the {\em integral definition}.
We shall refer to \cite{confronto} for a nice comparison between
these two different notions. In \cite{gruppone}, a problem like
\eqref{zwm=1} with $q=2$ is investigated, using the integral notion,
from the point of view of existence, nonexistence and regularity.
\vskip3pt \noindent To formulate the main result, we introduce
\be\label{zwm=6}
\Lambda_1:=\left(\frac{2^*_s-q}{2^*_s-2}(k_s\S(s,N))^{-\frac{q}{2}}
|\Omega|^{\frac{2^*_s-q}{2^*_s}}\right)^{-\frac{2}{2-q}}
\left[\frac{2-q}{2(2^*_s-q)}(k_s
\S_{s,\alpha,\beta})^{\frac{2^*_s}{2}}\right]^{\frac{2}{2^*_s-2}},
\ee where $|\Omega|$ is the Lebesgue measure of $\Omega$, $k_s$ is a
normalization constant  and $\S(s,N),\S_{s,\alpha,\beta}$ are best
Sobolev constants that will be introduced later. For $\gamma>0$, we
also consider
$$
{\mathscr C}_{\gamma}:=\big\{(\lambda,\mu)\in\R^2_+:\,  0<\lambda^{\frac{2}{2-q}}+\mu^{\frac{2}{2-q}}<\gamma\big\}.
$$
Then we have the following

\bt\label{th11}
The following facts holds
\begin{itemize}
\item[$(i)$] system \eqref{zwm=1} has at least one positive solution
for all $(\lambda,\mu)\in {\mathscr C}_{\Lambda_1}$.
\item[$(ii)$] there is $\Lambda_2<\Lambda_1$ such that
\eqref{zwm=1} has at least two positive solutions
for $(\lambda,\mu)\in {\mathscr C}_{\Lambda_2}$.
\end{itemize}
\et
\smallskip

\noindent Concerning regularity, one can get a priori estimates for
the solutions to \eqref{zwm=1} and hence obtain, as in
\cite[Proposition 5.2]{B1}, that $u,v\in
C^\infty(\overline{\Omega})$ for $s=1/2$, $u,v\in
C^{0,2s}(\overline{\Omega})$ if $0<s<1/2$ and $u,v\in
C^{1,2s-1}(\overline{\Omega})$ if $1/2<s<1$.

\vskip3pt
\noindent
The paper is organized as follows.\
In Section 2 we introduce the variational setting of the problem and present some preliminary results. In Section 3 we show that the
Palais-Smale condition holds for the energy
functional associated with  \eqref{zwm=1}
at energy levels in a suitable range related to the best
Sobolev constants. In Section 4 we give some properties about the Nehari manifold and  fibering maps. In Section 5 we investigate the existence of Palais-Smale
sequences. In Section 6 we obtain solutions to some related local minimization problems.
Finally, the proof  of Theorem \ref{th11} is given in Section 7.


\s{Some preliminary facts}

In this section, we collect some preliminary facts in order to establish the functional setting.
First of all, let us introduce the standard notations for future use in this paper. We denote the upper half-space in $\R^{N+1}_+$ by
$$
\R^{N+1}_{+}:=\{z=(x,y)=(x_1,\cdots,x_N,y)\in\R^{N+1}:y>0\}.
$$
Let $\Omega\subset\R^N$ be a smooth bounded domain.\ Denote by
$$
\mathcal
{C}_{\Omega}:=\Omega\times(0,\infty)\subset\R^{N+1}_{+},
$$
the cylinder with base $\Omega$ and its lateral boundary by
$\partial_L\mathcal {C}_{\Omega}:=\partial\Omega\times(0,\infty).$
The powers $(-\Delta)^s$ of the positive Laplace operator $-\Delta$, in $\Omega$, with zero Dirichlet boundary
conditions are defined via its spectral decomposition, namely
$$
(-\Delta)^su(x):=\sum_{j=1}^{\infty} a_j\rho_j^s\varphi_j(x),
$$
where $(\rho_j,\varphi_j)$ is the sequence of eigenvalues and eigenfunctions of  the
operator $-\Delta$ in $\Omega$ under zero Dirichlet boundary data
and $a_j$ are the coefficients of $u$ for the base $\{\varphi_j\}_{j=1}^{\infty}$
in $L^2(\Omega)$. In fact, the fractional Laplacian $(-\Delta)^{s}$ is well defined in the space of functions
$$
H^s_0(\Omega):=\Big\{u=\sum_{j=1}^{\infty}
a_j\varphi_j\in L^2(\Omega):~\|u\|_{H^s_0}=\Big(\sum_{j=1}^{\infty} a_j^2\rho_j^s \Big)^{1/2}<\infty\Big\},
$$
and $\|u\|_{H^s_0}=\|(-\Delta)^{s/2}u\|_{L^2(\Omega)}$. The dual space $H^{-s}(\Omega)$ is defined in the standard way, as well as the inverse operator $(-\Delta)^{-s}.$
\vskip0.1in
\bd\label{df21} We say that $(u,v)\in H^s_0(\Omega)\times
H^s_0(\Omega)$ is a solution of $(1.1)$ if the identity
\begin{align*}
& \int_{\Omega}\left((-\Delta)^{\frac{s}{2}}u(-\Delta)^{\frac{s}{2}}\varphi_1+(-\Delta)^{\frac{s}{2}}v
(-\Delta)^{\frac{s}{2}}\varphi_2\right)dx
-\int_{\Omega}(\lambda|u|^{q-2}u\varphi_1+\mu|v|^{q-2}v\varphi_2)dx \\
&-\frac{2\alpha}{\alpha+\beta}\int_{\Omega}|u|^{\alpha-2}u|v|^{\beta}\varphi_1dx
-\frac{2\beta}{\alpha+\beta}\int_{\Omega}|u|^{\alpha}|v|^{\beta-2}v\varphi_2dx=0,
\end{align*}
holds for all $(\varphi_1,\varphi_2)\in H^s_0(\Omega)\times
H^s_0(\Omega).$\ed

\noindent
Associated with problem \eqref{zwm=1}, we consider the energy functional
\begin{align*}
\mathcal{J}_{\lambda,\mu}(u,v) &:=\frac{1}{2}\int_{\Omega}\left(|(-\Delta)^{\frac{s}{2}}u|^2+|(-\Delta)^{\frac{s}{2}}v|^2\right)dx
-\frac{1}{q}\int_{\Omega}(\lambda|u|^q+\mu|v|^q)dx \\
&-\frac{2}{\alpha+\beta}\int_{\Omega}|u|^{\alpha}|v|^{\beta}dx.
\end{align*}
The functional is well defined in $H^s_0(\Omega)\times
H^s_0(\Omega)$, and moreover, the critical points of the functional
$\mathcal {J}_{\lambda,\mu}$ correspond to solutions of
\eqref{zwm=1}.\  We now conclude the main ingredients of a recently
developed technique used in order to deal with fractional powers of
the Laplacian.  To treat the nonlocal problem \eqref{zwm=1}, we
shall study a corresponding extension problem, which allows us to
investigate problem \eqref{zwm=1} by studying a local problem via
classical variational methods. We first define the extension
operator and fractional Laplacian for functions in
$H^s_0(\Omega)\times H^s_0(\Omega)$. We refer the
reader to \cite{B1,B2,BCP,capella} and to the references therein.

\bd\label{df2.2}  For a function $u\in H^s_0(\Omega)$, we denote  its  $s$-harmonic
 extension $w=E_{s}(u)$ to the cylinder $\mathcal
{C}_{\Omega}$ as the solution of the problem
 \begin{equation*}
\begin{cases}
            \mbox{\div}(y^{1-2s}\nabla w)=0 &\mbox{in}~ \mathcal
{C}_{\Omega}\vspace{0.1cm}\\
           w=0 &\mbox{on}~ \partial_L\mathcal
{C}_{\Omega}\vspace{0.1cm}\\
            w=u               &\mbox{on} ~\Omega\times\{0\},
     \end{cases}
     \end{equation*}
     and
\begin{equation*}
 (-\Delta)^su(x)=-k_{s}\lim_{y\rightarrow0^{+}}y^{1-2s}\frac{\partial
    w}{\partial y}(x,y),
\end{equation*}
    where $k_{s}=2^{1-2s}\Gamma(1-s)/\Gamma(s)$ is a normalization constant.\ed
    \noindent
The  extension  function $w(x,y)$ belongs to the space
 $$
 X^{s}_0(\mathcal {C}_{\Omega}):=
    \overline{C^{\infty}_0(\Omega\times[0,\infty))}^{\|\cdot\|_{X^{s}_0(\mathcal
    {C}_{\Omega})}}
$$  endowed with the norm
$$
\|z\|_{X^{s}_0(\mathcal
    {C}_{\Omega})}:=\left(k_{s}\int_{\mathcal
    {C}_{\Omega}}y^{1-2s}|\nabla z|^2dxdy\right)^{1/2}.
$$
 The    extension operator is an isometry between $H^s_0(\Omega)$ and
    $X^{s}_0(\mathcal {C}_{\Omega})$, namely
\begin{equation*}
\|u\|_{H^{s}_0(\Omega)}=\|E_s(u)\|_{X^{s}_0(\mathcal
    {C}_{\Omega})}, \quad\text{for all $u\in H^s_0(\Omega)$}.
    \end{equation*}
With this extension we can reformulate \eqref{zwm=1}
as the following local problem
\be\label{2.6}
\begin{cases}
            -\mbox{div}(y^{1-2s}\nabla w_1)=0,~~-\mbox{div}(y^{1-2s}\nabla w_2)=0 &\mbox{in}~ \mathcal
{C}_{\Omega}\vspace{0.1cm}\\
           w_1=w_2=0 &\mbox{on}~ \partial_L\mathcal
{C}_{\Omega}\vspace{0.1cm}\\

w_1=u,~~ w_2=v&\mbox{on}~{\Omega}\times\{0\}\vspace{0.1cm}\\
            \frac{\partial w_1}{\partial\nu^s}=\lambda|w_1|^{q-2}w_1+\frac{2\alpha}{\alpha+\beta}|w_1|^{\alpha-2}w_1|w_2|^{\beta}
                  &\mbox{on} ~\Omega\times\{0\}\vspace{0.1cm}\\
                   \frac{\partial
                   w_2}{\partial\nu^s}=\mu|w_2|^{q-2}w_2+\frac{2\beta}{\alpha+\beta}|w_1|^{\alpha}|w_2|^{\beta-2}w_2
                    &\mbox{on} ~\Omega\times\{0\},
 \end{cases}
 \ee
 where
 $$
 \frac{\partial
w_i}{\partial\nu^s}:=-k_{s}\lim_{y\rightarrow
0^+}y^{1-2s}\frac{\partial w_i}{\partial y}, \quad i=1,2,
$$
and $w_1,w_2\in X^{s}_0(\mathcal {C}_{\Omega})$ are the $s$-harmonic extension of $u,v\in
    H^s_0(\Omega)$, respectively.
    Let
    $$
    E^s_0(\mathcal {C}_{\Omega}):=X^{s}_0(\mathcal {C}_{\Omega})\times X^{s}_0(\mathcal {C}_{\Omega}).
    $$
    An energy solution to this problem is a function $(w_1,w_2)\in E^{s}_0(\mathcal
    {C}_{\Omega})$ satisfying
    $$\begin{array}{l}
    \displaystyle k_{s}\int_{\mathcal {C}_\Omega}y^{1-2s} \nabla
    w_1\cdot\nabla\varphi_1 dxdy+k_{s}\int_{\mathcal {C}_\Omega}y^{1-2s}\nabla
    w_2\cdot\nabla\varphi_2 dxdy\vspace{0.2cm}\\
    \hspace{1cm}=\displaystyle\lambda\int_{\Omega}|w_1|^{q-2}w_1\varphi_1dx+
    \frac{2\alpha}{\alpha+\beta}\int_{\Omega}|w_1|^{\alpha-2}w_1|w_2|^{\beta}\varphi_1dx\vspace{0.2cm}\\
    \hspace{1.4cm}\displaystyle
    +\mu\int_{\Omega}|w_2|^{q-2}w_2\varphi_2dx+\frac{2\beta}{\alpha+\beta}
    \int_{\Omega}|w_1|^{\alpha}|w_2|^{\beta-2}w_2\varphi_2dx,\end{array}$$
    for all $ (\varphi_1,\varphi_2)\in E^s_0(\mathcal
    {C}_{\Omega}).
    $   If  $(w_1,w_2)\in E^s_0(\mathcal
    {C}_{\Omega})$ satisfies \eqref{2.6}, then
    $(u,v)=(w_1(\cdot,0),w_2(\cdot,0))$, defined in the sense of traces,
    belongs to the space $H^s_0(\Omega)\times H^s_0(\Omega)$ and it is
    a solution of the original problem \eqref{zwm=1}.\
    The associated energy functional to the problem
     \eqref{2.6} is denoted by
\begin{align*}
\mathcal{I}_{\lambda,\mu}(w) &:=\mathcal {I}_{\lambda,\mu}(w_1,w_2)=\displaystyle\frac{k_s}{2}\int_{\mathcal
{C}_{\Omega}}y^{1-2s}(|\nabla w_1|^2+|\nabla w_2|^2)dxdy  \\
&-\frac{1}{q}
\int_{\Omega}(\lambda|w_1|^{q}+\mu
|w_2|^{q})dx-\frac{2}{\alpha+\beta}\int_{\Omega}|w_1|^{\alpha}|w_2|^{\beta}dx. \nonumber
\end{align*}
Critical points of $\mathcal {I}_{\lambda,\mu}$ in $E^s_0(\mathcal
{C}_{\Omega})$ correspond to critical points of $\mathcal
{J}_{\lambda,\mu}:H^s_0(\Omega)\times H^s_0(\Omega)\to\R$. In the
following lemma we list some relevant inequalities
from \cite{BCP}. \vskip0.1in

\bl\label{Lemma2.3}  For any $1\leq r \leq 2^*_s$ and any $z\in
X^s_0(\mathcal {C}_{\Omega}),$ it holds
\be\label{2.8}\left(\int_{\Omega}|u(x)|^rdx\right)^{\frac{2}{r}}\leq
C\int_{\mathcal {C}_{\Omega}}y^{1-2s}|\nabla z(x,y)|^2dxdy,\quad\,\,
u:={\rm Tr}(z), \ee for some positive constant $C=C(r,s,N,\Omega)$.\
Furthermore, the space $X^s_0(\mathcal {C}_{\Omega})$ is compactly
embedded into $L^r(\Omega)$, for every $r<2^*_s$. \el \vskip0.1in
\br\label{Remark2.4}\rm When $r=2^*_s,$ the best constant in
\eqref{2.8} is denoted by $\S(s,N)$, that is \be \label{2.9}
\S(s,N):=\inf_{z\in
X^s_0(\mathcal{C}_\Omega)\backslash\{0\}}\frac{\displaystyle\int_{\mathcal
{C}_{\Omega}}y^{1-2s}|\nabla z(x,y)|^2dxdy}
{\Big(\displaystyle\int_{\Omega}|z(x,0)|^{2^*_s}dx\Big)^{\frac{2}{2^*_s}}}.\ee
 It is not achieved in any bounded domain and, for all $ z\in X^s(\R^{N+1}_+)$,
\be
\label{2.10} \int_{\R^{N+1}_+}y^{1-2s}|\nabla z(x,y)|^2dxdy\geq \S(s,N)
\left(\int_{\R^{N}}|z(x,0)|^{\frac{2N}{N-2s}}dx\right)^{\frac{N-2s}{N}}.
\ee
$\S(s,N)$ is achieved for $\Omega=\R^N$ by functions  $w_{\varepsilon}$ which are the $s$-harmonic extensions of
\be
\label{2.11} u_{\varepsilon}(x):=\frac{\varepsilon^{(N-2s)/2}}{(\varepsilon^2+|x|^2)^{(N-2s)/2}},
\quad\,\, \varepsilon>0,\,\,\, x\in\R^N.
\ee Let $U(x)=(1+|x|^2)^{\frac{2s-N}{2}}$ and let $\mathcal{W}$ be the extension of $U$ (cf.\ \cite{BCP,B1}).\
Then
$$
\mathcal{W}(x,y)=E_s(U)=c_{N,s}y^{2s}\int_{\R^N}\frac{U(z)dz}{(|x-z|^2+y^2)^{\frac{N+2s}{2}}},
$$
is the extreme function for the fractional Sobolev inequality \eqref{2.10}.\
The constant $\S(s,N)$ given in \eqref{2.9} takes the exact value
$$
\S(s,N)=\frac{2\pi^s\Gamma(\frac{N+2s}{2})\Gamma(1-s)
    (\Gamma(\frac{N}{2}))^{\frac{2s}{N}}}{\Gamma(s)\Gamma(\frac{N-2s}{2})(\Gamma(N))^{s}},
$$
and it is achieved for $\Omega=\R^N$ by the functions
$w_{\varepsilon}$.
\er
\noindent
Now, we consider the following minimization problem
\be\label{2.12}
\S_{s,\alpha,\beta}:=\inf_{(w_1,w_2)\in
E^s_0(\mathcal {C}_\Omega)\backslash\{0\}}\frac{\displaystyle\int_{\mathcal
{C}_{\Omega}}y^{1-2s}(|\nabla w_1|^2+|\nabla w_2|^2)dxdy}
{\Big(\displaystyle\int_{\Omega}|w_1|^{\alpha}|w_2|^{\beta}dx\Big)^{\frac{2}{2^*_s}}}.
\ee
Using ideas from  \cite{Alves}, we establish a relationship between
$\S(s,N)$ and $\S_{s,\alpha,\beta}$ (see also \cite{gruppone}).

\bl\label{lemma2.5}
For the constants $\S(s,N)$ and $\S_{s,\alpha,\beta}$ introduced in \eqref{2.9} and \eqref{2.12}, it holds
\be\label{2.13}
\S_{s,\alpha,\beta}=
\left[\left(\frac{\alpha}{\beta}\right)^{\frac{\beta}{2^*_s}}+\left(\frac{\beta}{\alpha}\right)^{\frac{\alpha}{2^*_s}}
\right]\S(s,N).
\ee
In particular, the constant $\S_{s,\alpha,\beta}$ is achieved for $\Omega=\R^N$.
\el

\begin{proof}
Let $\{z_n\}\subset X^s_0(\mathcal {C}_{\Omega})$  be a
minimization sequence for $\S(s,N).$ Let $\sigma,t>0$ to be chosen later and consider
the sequences $w_{1,n}:=\sigma z_n$ and $w_{2,n}:=tz_n$ in $X^s_0(\mathcal {C}_{\Omega})$. By means of \eqref{2.12}, we have
\begin{equation*}
\frac{\sigma^2+t^2}{(\sigma^{\alpha}t^{\beta})^{\frac{2}{2^*_s}}}\frac{\displaystyle\int_{\mathcal
{C}_{\Omega}}y^{1-2s}|\nabla z_n(x,y)|^2dxdy}
{\left(\displaystyle\int_{\Omega}|z_n|^{2^*_s}dx\right)^{\frac{2}{2^*_s}}}\geq
\S_{s,\alpha,\beta}.
\end{equation*}
Defining $g:\R^+\to\R^+$ by setting $g(x):=x^{\frac{2\beta}{2^*_s}}+x^{\frac{-2\alpha}{2^*_s}}$, we have
$$
\frac{\sigma^2+t^2}{(\sigma^{\alpha}t^{\beta})^{\frac{2}{2^*_s}}}=g\Big(\frac{\sigma}{t}\Big),
\qquad\,\,
\min_{\R^+}g=g(x_0)=g(\sqrt{\alpha/\beta})=
\left(\frac{\alpha}{\beta}\right)^{\frac{\beta}{2^*_s}}+\left(\frac{\beta}{\alpha}\right)^{\frac{\alpha}{2^*_s}}.
$$
Choosing $\sigma,t$ in the previous inequality such that $\sigma/t=\sqrt{\alpha/\beta}$ and letting $n\to\infty$ yields
\begin{equation*}
\Big[\left(\frac{\alpha}{\beta}\right)^{\frac{\beta}{2^*_s}}+
\left(\frac{\beta}{\alpha}\right)^{\frac{\alpha}{2^*_s}}\Big]\S(s,N)\geq
\S_{s,\alpha,\beta}.
\end{equation*}
On the other hand, let $\{(w_{1,n},w_{2,n})\}\subset E^s_0(\mathcal {C}_\Omega)\setminus\{0\}$ be
a minimizing sequence for $\S_{s,\alpha,\beta}.$ Set
$h_n:=\sigma_nw_{2,n}$ for $\sigma_n>0$ with $\int_{\Omega}|w_{1,n}|^{2^*_s}dx=\int_{\Omega}|h_n|^{2^*_s}dx.$
Then Young's inequality yields
\begin{equation*}
\int_{\Omega}|w_{1,n}|^{\alpha}|h_n|^{\beta}dx\leq\displaystyle\frac{\alpha}{2^*_s}
\int_{\Omega}|w_{1,n}|^{2^*_s}dx+\frac{\beta}{2^*_s}
\int_{\Omega}|h_n|^{2^*_s}dx=\displaystyle\int_{\Omega}|h_n|^{2^*_s}dx=\int_{\Omega}|w_{1,n}|^{2^*_s}dx.
\end{equation*}
In turn, we can estimate
$$
\begin{array}{l}
\displaystyle\frac{\displaystyle\int_{\mathcal {C}_{\Omega}}y^{1-2s}(|\nabla
w_{1,n}(x,y)|^2+|\nabla w_{2,n}(x,y)|^2)dxdy}
{\Big(\displaystyle\int_{\Omega}|w_{1,n}|^{\alpha}|w_{2,n}|^{\beta}dx\Big)^{\frac{2}{\alpha+\beta}}}\vspace{0.1cm}\\
=\displaystyle
\frac{\sigma_n^{\frac{2\beta}{2^*_s}}\displaystyle\int_{\mathcal
{C}_{\Omega}}y^{1-2s}(|\nabla w_{1,n}(x,y)|^2+|\nabla
w_{2,n}(x,y)|^2)dxdy}
{\Big(\displaystyle\int_{\Omega}|w_{1,n}|^{\alpha}|h_n|^{\beta}dx\Big)^{\frac{2}{\alpha+\beta}}}\vspace{0.1cm}\\
\geq\displaystyle\frac{\sigma_n^{\frac{2\beta}{2^*_s}}\displaystyle\int_{\mathcal
{C}_{\Omega}}y^{1-2s}|\nabla w_{1,n}(x,y)|^2dxdy}
{\left(\displaystyle\int_{\Omega}|w_{1,n}|^{2^*_s}dx\right)^{2/2^*_s}}\vspace{0.1cm}\\
\displaystyle\hspace{0.3cm}+\frac{\sigma_n^{\frac{2\beta}{2^*_s}}\sigma_n^{-2}
\displaystyle\int_{\mathcal {C}_{\Omega}}y^{1-2s}|\nabla h_n(x,y)|^2dxdy}
{\left(\displaystyle\int_{\Omega}|h_n|^{2^*_s}dx\right)^{2/2^*_s}}\vspace{0.1cm}\\
\geq \S(s,N)g(\sigma_n)\geq\displaystyle \S(s,N)g(\sqrt{\alpha/\beta}).
\end{array}
$$
Passing to the limit as $n\to\infty$ in the last inequality we obtain
\begin{equation*}
\Big[\left(\frac{\alpha}{\beta}\right)^{\frac{\beta}{2^*_s}}+
\left(\frac{\beta}{\alpha}\right)^{\frac{\alpha}{2^*_s}}\Big]\S(s,N)\leq
\S_{s,\alpha,\beta}.
\end{equation*}
Whence, the conclusion follows by combining the previous inequalities.
\end{proof}

\noindent
In the end of this section, we fix some notations that will be used in the sequel.

\vskip0.1in
\par\noindent {\it Notations.} In this paper we use the following notations:
\begin{itemize}
\item $L^p(\Omega)$, $1\leq p\leq\infty$ denote Lebesgue spaces, with
norm $\|\cdot\|_p$. $E=X^{s}_0(\mathcal {C}_{\Omega})\times X^{s}_0(\mathcal {C}_{\Omega})$ is equipped with the norm $\|z\|^2=\|(w_1,w_2)\|^2=\|w_1\|_{X^s_0(\mathcal {C}_{\Omega})}^2+\|w_2\|_{X^s_0(\mathcal {C}_{\Omega})}^2.$

\item  The dual space of a Banach space $E$ will be denoted by $E^{-1}.$
We set $ tz=t(w_1,w_2)=(tw_1,tw_2)$ for all $z\in E$ and $t\in\R.$ $z=(w_1,w_2)$ is said to be non-negative in $\mathcal{C}_\Omega$
if $w_1(x,y)\geq0, w_2(x,y)\geq0$ in $\mathcal{C}_\Omega$ and to be
positive if $w_1(x,y)>0, w_2(x,y)>0$ in $\mathcal{C}_\Omega$.
\item $|\Omega|$ is the Lebesgue measure of $\Omega$. $B(0;r)$ is  the ball at the origin with radius $r$.
\item $\O(\varepsilon^t)$ denotes $|\O(\varepsilon^t)|/\varepsilon^t\leq C$ as $\varepsilon\rightarrow0$ for $t\geq0.$ $o_n(1)$ denotes $o_n(1)\rightarrow0$ as $n\rightarrow\infty.$
\item $C, C_i, c$ will denote various positive constants which may vary from line to line.
\end{itemize}


\s{The Palais-Smale condition}

In this section we shall detect the range of values $c$ for which the
$(PS)_c$-condition holds for the functional  $\mathcal
{I}_{\lambda,\mu}$.\ Let $c\in\R$ and set, for simplicity, $E:=E^s_0(\mathcal {C}_{\Omega})$. We say $\{z_n\}\subset E$ is a
$(PS)_c$-sequence in $E$ for $\mathcal {I}_{\lambda,\mu}$ if
$\mathcal {I}_{\lambda,\mu}(z_n)=c+o_n(1)$ and $\mathcal
{I}'_{\lambda,\mu}(z_n)=o_n(1)$ strongly in $E^{-1}$, as
$n\to\infty.$ If any $(PS)_c$-sequence $\{z_n\}$ in $E$ for
$\mathcal {I}_{\lambda,\mu}$ admits a convergent subsequence, we say
that $\mathcal {I}_{\lambda,\mu}$ satisfies the $(PS)_c$-condition.
We shall need the following preliminary result.

\vskip0.1in
\bl\label{Lemma3.1}
Let $\{z_n\}\subset E$ be a $(PS)_c$-sequence for $\mathcal {I}_{\lambda,\mu}$ for some $c\in\R$ with
$z_n\rightharpoonup z$ in $E$. Then  $\mathcal
{I}'_{\lambda,\mu}(z)=0$ and there exists a positive constant $K_0$,
depending only on $q,N,s$ and $|\Omega|$, such that
$$
\mathcal
{I}_{\lambda,\mu}(z)\geq-K_0\Big(\lambda^{\frac{2}{2-q}}+\mu^{\frac{2}{2-q}}\Big).
$$
\el
\begin{proof}
Consider $z_n=(w_{1,n},w_{2,n})\subset E$ and $z=(w_1,w_2)\in E$. If $\{z_n\}$ is a
$(PS)_c$-sequence for $\mathcal {I}_{\lambda,\mu}$ with
$z_n\rightharpoonup z$ in $E$, then $w_{1,n}\rightharpoonup w_{1}$ and
$w_{2,n}\rightharpoonup w_{2}$ in $X^{s}_0(\mathcal {C}_{\Omega}),$ as
$n\to\infty.$  Then, by virtue of Sobolev embedding theorem (Lemma
\ref{Lemma2.3}),  we also have $w_{1,n}(\cdot,0)\rightarrow w_{1}(\cdot,0)$ and
$w_{2,n}(\cdot,0)\to w_{2}(\cdot,0)$ strongly in
$L^q(\Omega)$, as $n\to\infty$.\  Of course, up to a further subsequence,
$w_{1,n}(\cdot,0)\to w_{1}(\cdot,0)$ and $w_{2,n}(\cdot,0)\rightarrow w_{2}(\cdot,0)$
a.e.\ in $\Omega.$  It is standard to check  that  $\mathcal {I}'_{\lambda,\mu}(z)=0$.
This implies that $\langle\mathcal {I}'_{\lambda,\mu}(z),z\rangle=0$, namely
$$
k_s\int_{\mathcal{C}_\Omega}y^{1-2s}\left(|\nabla w_1|^{2}+|\nabla w_2|^{2}\right)dxdy=\int_{\Omega}(\lambda|w_1|^q+\mu|w_2|^q)dx
+2\int_{\Omega}|w_1|^{\alpha}|w_2|^{\beta}dx.
$$
Consequently, we get
\begin{align}
\label{3.1}
\mathcal{I}_{\lambda,\mu}(z)&=\Big(\frac{1}{2}-\frac{1}{2^*_s}\Big)k_s\int_{\mathcal{C}_\Omega}y^{1-2s}(|\nabla
w_1|^{2}+|\nabla w_2|^{2})dxdy \\
&-\Big(\frac{1}{q}-\frac{1}{2^*_s}\Big)\int_{\Omega}(\lambda|w_1|^q+\mu|w_2|^q)dx.  \notag
\end{align}
By \eqref{3.1}, H\"{o}lder and Young inequalities and the
Sobolev embedding theorem, we obtain
 \begin{align*}
\mathcal {I}_{\lambda,\mu}(z)&=\displaystyle\Big(\frac{1}{2}-\frac{1}{2^*_s}\Big)\|z\|^2
-\Big(\frac{1}{q}-\frac{1}{2^*_s}\Big)\int_{\Omega}(\lambda|w_1|^q+\mu|w_2|^q)dx    \\
&\geq\displaystyle\frac{s}{N}\|z\|^2-\frac{2^*_s-q}{q2^*_s}|\Omega|^{(2^*_s-q)/{2^*_s}}
(\lambda\|w_1\|_{2^*_s}^q+\mu\|w_2\|_{2^*_s}^q)  \\
&\geq\displaystyle\frac{s}{N}\|z\|^2-\frac{2^*_s-q}{q2^*_s}|\Omega|^{(2^*_s-q)/{2^*_s}}(k_s\S(s,N))^{-\frac{q}{2}}
(\lambda\|w_1\|^q+\mu\|w_2\|^q)  \\
&\geq\displaystyle\frac{s}{N}\|z\|^2-C\Big(\frac{2-q}{2}\varepsilon\left[\lambda^{\frac{2}{2-q}}+\mu^{\frac{2}{2-q}}\right]
+\frac{q}{2}\varepsilon^{-\frac{2-q}{q}}\left[\|w_1\|^2+||w_2\|^2\right]\Big)   \\
&=\displaystyle\frac{s}{N}\|z\|^2-\frac{s}{N}\|z\|^2-K_0\left(\lambda^{\frac{2}{2-q}}+\mu^{\frac{2}{2-q}}\right)   \\
&=\displaystyle-K_0\left(\lambda^{\frac{2}{2-q}}+\mu^{\frac{2}{2-q}}\right),
\end{align*}
which yields the assertion, where we have put
$$
C:=\frac{2^*_s-q}{q2^*_s}|\Omega|^{(2^*_s-q)/{2^*_s}}(k_s\S(s,N))^{-\frac{q}{2}},
\quad
\varepsilon:=\left(\frac{NqC}{2s}\right)^{\frac{q}{2-q}},
\quad
K_0:=\frac{\varepsilon(2-q)}{2}C,
$$
the  positive constants involving only $q$, $|\Omega|$, $s$ and $N$.
\end{proof}

\vskip0.1in
\bl\label{Lemma3.2} If $\{z_n\}\subset E$ is a
$(PS)_c$-sequence for $\mathcal {I}_{\lambda,\mu}$, then $\{z_n\}$
is bounded  in $E$.\el

\begin{proof}Let $z_n=(w_{1,n},w_{2,n})\subset E$
be a $(PS)_c$-sequence for $\mathcal {I}_{\lambda,\mu}$ and suppose, by contradiction, that
$\|z_n\|\rightarrow\infty$, as $n\to\infty$. Put
$$
\widetilde{z}_n=(\widetilde{w}_{1,n},\widetilde{w}_{2,n}):=\frac{z_n}{\|z_n\|}=\Big(\frac{w_{1,n}}{\|z_n\|},
\frac{w_{2,n}}{\|z_n\|}\Big).
$$
We may assume that
$\widetilde{z}_n\rightharpoonup\widetilde{z}=(\widetilde{w}_1,\widetilde{w}_2)$ in $E$.
This implies that $\widetilde{w}_{1,n}(\cdot,0)\to\widetilde{w}_1(\cdot,0)$
and $\widetilde{w}_{2,n}(\cdot,0)\to\widetilde{w}_2(\cdot,0)$
strongly in $L^r(\Omega)$ for all $1\leq r<2^*_s$ and, thus,
$$
\int_{\Omega}(\lambda|\widetilde{w}_{1,n}|^q+\mu|\widetilde{w}_{2,n}|^q)dx=
\int_{\Omega}(\lambda|\widetilde{w}_1|^q+\mu|\widetilde{w}_2|^q)dx+o_n(1).
$$
Since $\{z_n\}$ is a $(PS)_c$ sequence for $\mathcal{I}_{\lambda,\mu}$ and $\|z_n\|\rightarrow\infty,$ we get
\begin{align}
\label{3.3}
& \frac{k_s}{2}\int_{\mathcal{C}_\Omega}y^{1-2s}\left(|\nabla \widetilde{w}_{1,n}|^{2}+|\nabla \widetilde{w}_{2,n}|^{2}\right)dxdy
-\frac{\|z_n\|^{q-2}}{q}\int_{\Omega}(\lambda|\widetilde{w}_{1,n}|^q+ \mu|\widetilde{w}_{2,n}|^q)dx  \\
&-\frac{2\|z_n\|^{2^*_s-2}}{2^*_s}
\int_{\Omega}|\widetilde{w}_{1,n}|^{\alpha}|\widetilde{w}_{2,n}|^{\beta}dx=o_n(1),  \notag
\end{align}
and
\begin{align}\label{3.4}
& k_s\int_{\mathcal{C}_\Omega}y^{1-2s}\left(|\nabla \widetilde{w}_{1,n}|^{2}+|\nabla \widetilde{w}_{2,n}|^{2}\right)dxdy
-\|z_n\|^{q-2}\int_{\Omega}(\lambda|\widetilde{w}_{1,n}|^q+\mu|\widetilde{w}_{2,n}|^q)dx \\
& -2\|z_n\|^{2^*_s-2}\int_{\Omega}|\widetilde{w}_{1,n}|^{\alpha}|\widetilde{w}_{2,n}|^{\beta}dx
=o_n(1).     \notag
\end{align}
 Combining \eqref{3.3} and \eqref{3.4}, as $n\to\infty$, we obtain
\begin{align}
\label{3.5}
& k_s\int_{\mathcal{C}_\Omega}y^{1-2s}\left(|\nabla
\widetilde{w}_{1,n}|^{2}+|\nabla
\widetilde{w}_{2,n}|^{2}\right)dxdy   \\
&   =\frac{2(2^*_s-q)}{q(2^*_s-2)}\|z_n\|^{q-2}\int_{\Omega}(\lambda|\widetilde{w}_{1,n}|^q
+\mu|\widetilde{w}_{2,n}|^q)dx+o_n(1).   \notag
\end{align}
In view of $1<q<2$ and $\|z_n\|\rightarrow\infty,$ \eqref{3.5} implies that
$$
k_s\int_{\mathcal{C}_\Omega}y^{1-2s}\left(|\nabla \widetilde{w}_{1,n}|^{2}+|\nabla \widetilde{w}_{2,n}|^{2}\right)dxdy\to 0,
$$
as $n\to\infty$, which contradicts to the fact that $\|\widetilde{z}_n\|=1$ for any $n\geq 1$.
\end{proof}

\bl\label{Lemma3.3}  $\mathcal {I}_{\lambda,\mu}$
satisfies the $(PS)_c$ condition with $c$ satisfying
$$
-\infty<c<c_{\infty}:=\frac{2s}{N}\left(\frac{k_s\S_{s,\alpha,\beta}}{2}\right)^{\frac{N}{2s}}
-K_0\left(\lambda^{\frac{2}{2-q}}+\mu^{\frac{2}{2-q}}\right),
$$ where $K_0$ is the positive constant
introduced in Lemma \ref{Lemma3.1}
\el

\begin{proof}
    Let $\{z_n\}\subset E$ be a $(PS)_c$-sequence for $\mathcal
{I}_{\lambda,\mu}$ with $c\in(-\infty,c_{\infty})$. Write
$z_n=(w_{1,n},w_{2,n}).$ By Lemma \ref{Lemma3.2},  we see that $\{z_n\}$ is
bounded in $E$ and $z_n\rightharpoonup z=(w_{1},w_2)$ up to a
subsequence and $z$  is a critical point of $\mathcal
{I}_{\lambda,\mu}$. Furthermore,
$w_{1,n}\rightharpoonup w_{1}$ and
$w_{2,n}\rightharpoonup w_2$ weakly in
$X^s_0(\mathcal{C}_\Omega)$, $w_{1,n}(\cdot,0)\rightarrow
w_1(\cdot,0)$ and $w_{2,n}(\cdot,0)\rightarrow w_2(\cdot,0)$ strongly in
$L^r(\Omega)$ for every $1\leq r<2^*_s$ and
$w_{1,n}(\cdot,0)\rightarrow
w_1(\cdot,0),w_{2,n}(\cdot,0)\rightarrow w_2(\cdot,0)$ a.e.\ in
$\Omega$, up to a subsequence. Hence, we have
\be\label{3.6}\int_{\Omega}(\lambda|{w}_{1,n}|^q+\mu|{w}_{2,n}|^q)dx=\int_{\Omega}(\lambda|{w_1}|^q+\mu|{w_2}|^q)dx+o_n(1).\ee
Let $\widehat{w}_{1,n}:=w_{1,n}-w_1$, $\widehat{w}_{2,n}:=w_{2,n}-w_2$ and
$\widehat{z}_n:=(\widehat{w}_{1,n},\widehat{w}_{2,n}).$ Then, we obtain
\begin{equation*}
\|\widehat{z}_n\|^2=\|z_n\|^2-\|z\|^2+o_n(1).
\end{equation*}
In light of \cite[Lemma 2.1]{Han}, we also get
\be\label{3.8}\int_{\Omega}|\widehat{w}_{1,n}|^{\alpha}|\widehat{w}_{2,n}|^{\beta}dx=
\int_{\Omega}|w_{1,n}|^{\alpha}|w_{2,n}|^{\beta}dx
-\int_{\Omega}|w_1|^{\alpha}|w_2|^{\beta}dx+o_n(1).\ee Using $\mathcal
{I}_{\lambda,\mu}(z_n)=c+o_n(1)$ and $\mathcal
{I}'_{\lambda,\mu}(z_n)=o_n(1)$ and  \eqref{3.6}-\eqref{3.8}, we conclude
\begin{align}
\label{3.9}\frac{1}{2}\|\widehat{z}_n\|^2-\frac{2}{2^*_s}
\int_{\Omega}|\widehat{w}_{1,n}|^{\alpha}|\widehat{w}_{2,n}|^{\beta}dx&=c-\mathcal
{I}_{\lambda,\mu}(z)+o_n(1),  \\
\|\widehat{z}_n\|^2-2
\int_{\Omega}|\widehat{w}_{1,n}|^{\alpha}|\widehat{w}_{2,n}|^{\beta}dx&=\langle\mathcal
{I}'_{\lambda,\mu}(z_n),z_n\rangle-\langle\mathcal
{I}'_{\lambda,\mu}(z),z\rangle+o_n(1)=o_n(1).  \notag
\end{align}
Hence, we may assume that
\be\label{3.10}\|\widehat{z}_n\|^2\rightarrow \ell, \qquad
2\int_{\Omega}|\widehat{w}_{1,n}|^{\alpha}|\widehat{w}_{2,n}|^{\beta}dx\rightarrow
\ell.\ee
If $\ell=0,$ the proof is complete. If $\ell>0$ then from
\eqref{3.10} and the definition of $\S_{s,\alpha,\beta}$, we have
$$
k_s\S_{s,\alpha,\beta}\left(\frac{\ell}{2}\right)^{\frac{2}{2^*_s}}=\displaystyle
k_s\S_{s,\alpha,\beta}\lim_{n\rightarrow\infty}
\left(\int_{\Omega}|\widehat{w}_{1,n}|^{\alpha}|\widehat{w}_{2,n}|^{\beta}dx\right)^{\frac{2}{2^*_s}}
\leq\displaystyle\lim_{n\rightarrow\infty}\|\widehat{z}_n\|^2=\ell,
$$
which implies that $\ell\geq 2(k_s\S_{s,\alpha,\beta}/2)^{\frac{N}{2s}}.$
On the other hand, from Lemma \ref{Lemma3.1}, \eqref{3.9} and \eqref{3.10}, 
$$
c=\left(\frac{1}{2}-\frac{1}{2^*_s}\right)\ell+\mathcal
{I}_{\lambda,\mu}(z)\geq\frac{2s}{N}\left(\frac{k_s\S_{s,\alpha,\beta}}{2}\right)^{\frac{N}{2s}}
-K_0\left(\lambda^{\frac{2}{2-q}}+\mu^{\frac{2}{2-q}}\right), $$which
contradicts $c<c_\infty.$
\end{proof}

\s{The Nehari manifold}
Since the energy functional $\mathcal{I}_{\lambda,\mu}$ associated with \eqref{2.6} is not bounded on $E$, it is useful to consider
the functional on the Nehari manifold
$$
\mathcal{N}_{\lambda,\mu}:=\big\{z\in E\backslash\{0\}:\, \langle\mathcal
{I}'_{\lambda,\mu}(z),z\rangle=0\big\}.
$$
Thus, $z=(w_1,w_2)\in \mathcal{N}_{\lambda,\mu}$ if and only if $z\neq 0$ and
\be\label{4.1}\langle\mathcal
{I}'_{\lambda,\mu}(z),z\rangle=\|z\|^2-Q_{\lambda,\mu}(z)
-2\int_{\Omega}|w_1|^{\alpha}|w_2|^{\beta}dx=0,
\ee
 where
 $$
 Q_{\lambda,\mu}(z):=\int_{\Omega}(\lambda|w_1|^q+\mu|w_2|^q)dx.
 $$
 It is clear that all critical points of $\mathcal{I}_{\lambda,\mu}$ must lie on $\mathcal{N}_{\lambda,\mu}$ and, as we will see below, local minimizers on $\mathcal{N}_{\lambda,\mu}$ are actually critical points of $\mathcal
{I}_{\lambda,\mu}.$ We have the following results.

\bl\label{Lemma4.1} The energy functional $\mathcal
{I}_{\lambda,\mu}$ is bounded below and coercive on
$\mathcal{N}_{\lambda,\mu}$.\el

\begin{proof}Let $z=(w_1,w_2)\in \mathcal{N}_{\lambda,\mu}.$ Then by
\eqref{4.1} and the H\"{o}lder and Sobolev inequalities
\begin{align}
\label{4.2}
\mathcal{I}_{\lambda,\mu}(z)&=\displaystyle\frac{2^*_s-2}{22^*_s}\|z\|^2-\frac{2^*_s-q}{q2^*_s}Q_{\lambda,\mu}(z)  \\
&\geq\displaystyle\frac{2^*_s-2}{22^*_s}\|z\|^2-\frac{2^*_s-q}{q2^*_s}
(k_s\S(s,N))^{-\frac{q}{2}}|\Omega|^{\frac{2^*_s-q}{2^*_s}}\left(\lambda^{\frac{2}{2-q}}
+\mu^{\frac{2}{2-q}}\right)^{\frac{2-q}{2}}\|z\|^q. \notag
\end{align}
Since $1<q<2$, the functional $\mathcal {I}_{\lambda,\mu}$ is coercive and bounded
below on $\mathcal{N}_{\lambda,\mu}.$
\end{proof}
\noindent
 The Nehari manifold $\mathcal{N}_{\lambda,\mu}$ is closely linked to
 the fibering map $\Phi_z: t\rightarrow\mathcal{I}_{\lambda,\mu}(tz)$ given by
$$\Phi_z(t):=\mathcal {I}_{\lambda,\mu}(tz)=\frac{t^2}{2}\|z\|^2-\frac{t^q}{q}
\int_{\Omega}(\lambda|w_1|^{q}+\mu
|w_2|^{q})dx-\frac{2t^{2^*_s}}{\alpha+\beta}\int_{\Omega}|w_1|^{\alpha}|w_2|^{\beta}dx.$$
Such maps were introduced by Drabek and Pohozaev in  \cite{DP} and later
on used by Brown and Zhang  \cite{BZ}. Notice that we have
\begin{align*}
\Phi'_z(t)&=t\|z\|^2-t^{q-1}
\int_{\Omega}(\lambda|w_1|^{q}+\mu
|w_2|^{q})dx-2t^{2^*_s-1}\int_{\Omega}|w_1|^{\alpha}|w_2|^{\beta}dx, \\
\Phi''_z(t)&=\|z\|^2-(q-1)t^{q-2} \int_{\Omega}(\lambda|w_1|^{q}+\mu
|w_2|^{q})dx-2(2^*_s-1)t^{2^*_s-2}\int_{\Omega}|w_1|^{\alpha}|w_2|^{\beta}dx. \notag
\end{align*}
It is clear that $\Phi'_z(t)=0$ if and only if
$tz\in\mathcal{N}_{\lambda,\mu}.$ Hence,
$z\in\mathcal{N}_{\lambda,\mu}$ if and only if $\Phi'_z(1)=0$.
Introduce now the functional
$$
\mathcal
{R}_{\lambda,\mu}(z):=\langle\mathcal
{I}'_{\lambda,\mu}(z),z\rangle.
$$ Then, for every $z\in\mathcal{N}_{\lambda,\mu}$, we have
\begin{align}
\label{4.5}
\langle\mathcal
{R}'_{\lambda,\mu}(z),z\rangle&=\displaystyle 2\|z\|^2
-qQ_{\lambda,\mu}(z)-22^*_s\int_{\Omega}|w_1|^{\alpha}|w_2|^{\beta}dx  \notag \\
&=\displaystyle (2-q)\|z\|^2-2(2^*_s-q)\int_{\Omega}|w_1|^{\alpha}|w_2|^{\beta}dx  \\
&=\displaystyle(2^*_s-q)Q_{\lambda,\mu}(z)-(2^*_s-2)\|z\|^2.  \notag
\end{align}
Following the method used in  \cite{Wu1}, we split
$\mathcal{N}_{\lambda,\mu}$ into three parts
\begin{align*}
\mathcal{N}^{+}_{\lambda,\mu}&:=\{z\in \mathcal{N}_{\lambda,\mu}: \langle\mathcal
{R}'_{\lambda,\mu}(z),z\rangle>0\},  \\
\mathcal{N}^{0}_{\lambda,\mu}&:=\{z\in \mathcal{N}_{\lambda,\mu}:\langle\mathcal
{R}'_{\lambda,\mu}(z),z\rangle=0\},  \\
\mathcal{N}^{-}_{\lambda,\mu}&:=\{z\in \mathcal{N}_{\lambda,\mu}:\langle\mathcal
{R}'_{\lambda,\mu}(z),z\rangle<0\}.
\end{align*}
Then, we have the following lemmas.

\bl\label{Lemma4.2}  If $z_0$ is a local
minimizer for $\mathcal {I}_{\lambda,\mu}$ on
$\mathcal{N}_{\lambda,\mu}$ and
$z_0\not\in\mathcal{N}^{0}_{\lambda,\mu}$, then $\mathcal
{I}'_{\lambda,\mu}(z_0)=0.$
\el
\begin{proof}
    Let $z_0=(w_{0,1},w_{0,2})\in \mathcal{N}_{\lambda,\mu}$ be a local minimizer for
    the functional $\mathcal{I}_{\lambda,\mu}$ on $\mathcal{N}_{\lambda,\mu}$.
Hence, there exists a Lagrange multiplier $\gamma\in\R$ such that
$\mathcal {I}'_{\lambda,\mu}(z_0)=\gamma\mathcal
{R}'_{\lambda,\mu}(z_0)$. Thus,
\begin{equation*}
\langle\mathcal {I}'_{\lambda,\mu}(z_0),z_0\rangle=\gamma\langle\mathcal
{R}'_{\lambda,\mu}(z_0),z_0\rangle=0.
\end{equation*}
Since $z_0\not\in\mathcal{N}^{0}_{\lambda,\mu},$ then
$\langle\mathcal {R}'_{\lambda,\mu}(z_0),z_0\rangle\neq0,$ yielding
$\gamma=0.$ This completes the proof.
\end{proof}

\noindent
Let $\Lambda_1$ be the positive number defined in \eqref{zwm=6}. Then we have the following result.

\bl\label{Lemma4.3}
Assume that $(\lambda,\mu)\in {\mathscr C}_{\Lambda_1}$. Then
$\mathcal{N}^{0}_{\lambda,\mu}=\emptyset$.
\el
\begin{proof}
Assume by contradiction that there exist $\lambda>0$ and $\mu>0$
with $0<\lambda^{\frac{2}{2-q}}+\mu^{\frac{2}{2-q}}<\Lambda_1$ and such
that $\mathcal{N}^{0}_{\lambda,\mu}\neq\emptyset$. Let $z\in \mathcal{N}^{0}_{\lambda,\mu}$.
Then, by virtue of  \eqref{4.5}, we get
$$
\|z\|^2=\frac{2(2^*_s-q)}{2-q}\int_{\Omega}|w_1|^{\alpha}|w_2|^{\beta}dx,\qquad
\|z\|^2=\frac{2^*_s-q}{2^*_s-2}Q_{\lambda,\mu}(z).
$$
By H\"{o}lder inequality and the Sobolev embedding theorem, we have
\begin{align*}
\|z\| &\geq\Big[\frac{2-q}{2(2^*_s-q)}(k_s\S_{s,\alpha,\beta})^{\frac{2^*_s}{2}}\Big]^{\frac{1}{2^*_s-2}},  \\
\|z\| &\leq\Big(\frac{2^*_s-q}{2^*_s-2}(k_s\S(s,N))^{-\frac{q}{2}}|\Omega|^{\frac{2^*_s-q}{2^*_s}}\Big)^{\frac{1}{2-q}}
\big(\lambda^{\frac{2}{2-q}}+\mu^{\frac{2}{2-q}}\big)^{\frac{1}{2}},
\end{align*}
which leads to the inequality
$$
\lambda^{\frac{2}{2-q}}+\mu^{\frac{2}{2-q}}\geq
\Big(\frac{2^*_s-q}{2^*_s-2}(k_s\S(s,N))^{-\frac{q}{2}}
|\Omega|^{\frac{2^*_s-q}{2^*_s}}\Big)^{-\frac{2}{2-q}}
\Big[\frac{2-q}{2(2^*_s-q)}(k_s\S_{s,\alpha,\beta})^{\frac{2^*_s}{2}}\Big]^{\frac{2}{2^*_s-2}}=\Lambda_1,
$$
contradicting the assumption.
\end{proof}

\noindent
From Lemma~\ref{Lemma4.3}, if $0<\lambda^{\frac{2}{2-q}}+\mu^{\frac{2}{2-q}}<\Lambda_1$, we can write $\mathcal{N}_{\lambda,\mu}=
\mathcal{N}^{+}_{\lambda,\mu}\cup\mathcal{N}^{-}_{\lambda,\mu}$ and
define
$$
\alpha_{\lambda,\mu}:=\inf_{z\in\mathcal{N}_{\lambda,\mu}}\mathcal
{I}_{\lambda,\mu}(z),
\qquad
\alpha^{+}_{\lambda,\mu}:=\inf_{z\in\mathcal{N}^{+}_{\lambda,\mu}}\mathcal
{I}_{\lambda,\mu}(z),
\qquad
\alpha^{-}_{\lambda,\mu}:=\inf_{z\in\mathcal{N}^{-}_{\lambda,\mu}}\mathcal
{I}_{\lambda,\mu}(z).
$$
Moreover, we have the following properties about the Nehari manifold $\mathcal{N}_{\lambda,\mu}.$

\bt\label{Theorem4.4}
The following facts holds
\begin{itemize}
\item[$(i)$] If $(\lambda,\mu)\in {\mathscr C}_{\Lambda_1}$,
then we have $\alpha_{\lambda,\mu}\leq\alpha^{+}_{\lambda,\mu}<0;$
\item[$(ii)$] If $(\lambda,\mu)\in {\mathscr C}_{(q/2)^{2/(2-q)}\Lambda_1}$,
then we have $\alpha^{-}_{\lambda,\mu}>c_0$ for some positive
constant $c_0$ depending on $\lambda,\mu,N,s$ and $|\Omega|.$
\end{itemize}
\et

\begin{proof}
(i) Let $z=(w_1,w_2)\in\mathcal{N}^{+}_{\lambda,\mu}.$ By formula \eqref{4.5}, we have
$$\frac{2-q}{2(2^*_s-q)}\|z\|^2>\int_{\Omega}|w_1|^{\alpha}|w_2|^{\beta}$$
and so,
\begin{align*}
\mathcal{I}_{\lambda,\mu}(z)&=\displaystyle\left(\frac{1}{2}-\frac{1}{q}\right)\|z\|^2
+2\left(\frac{1}{q}-\frac{1}{2^*_s}\right)\int_{\Omega}|w_1|^{\alpha}|w_2|^{\beta}dx\vspace{0.2cm}\\
&\leq\displaystyle\left[\left(\frac{1}{2}-\frac{1}{q}\right)+\left(\frac{1}{q}-\frac{1}{2^*_s}\right)
\frac{2-q}{2^*_s-q}\right]\|z\|^2\vspace{0.2cm}\\
&=\displaystyle-\frac{(2-q)s}{Nq}\|z\|^2<0.
\end{align*}
Therefore, by the definition of
$\alpha_{\lambda,\mu},\alpha^+_{\lambda,\mu}$, we can deduce that
$\alpha_{\lambda,\mu}\leq\alpha^{+}_{\lambda,\mu}<0.$

\par (ii) Let $z\in\mathcal{N}^{-}_{\lambda,\mu}$. By equation \eqref{4.5},
$$\frac{2-q}{2(2^*_s-q)}\|z\|^2<\int_{\Omega}|w_1|^{\alpha}|w_2|^{\beta}dx.$$
By the H\"{o}lder inequality and the Sobolev embedding theorem, we
have
$$\int_{\Omega}|w_1|^{\alpha}|w_2|^{\beta}dx\leq
(k_s\S_{s,\alpha,\beta})^{-\frac{2^*_s}{2}}\|z\|^{2^*_s}.$$
Hence, we obtain
\begin{equation*}
\|z\|>\left(\frac{2-q}{2(2^*_s-q)}\right)^{\frac{1}{2^*_s-2}}(k_s\S_{s,\alpha,\beta})^{\frac{N}{4s}},\quad\text{for
all $z\in\mathcal{N}^{-}_{\lambda,\mu}$.}
\end{equation*}
From the last inequality  we infer that
\begin{align*}
\mathcal{I}_{\lambda,\mu}(z)&=\displaystyle\frac{2^*_s-2}{22^*_s}\|z\|^2-\frac{2^*_s-q}{q2^*_s}Q_{\lambda,\mu}(z)  \\
&\geq\|z\|^q\left[\frac{2^*_s-2}{22^*_s}\|z\|^{2-q}-\frac{2^*_s-q}{q2^*_s}(k_s\S(s,N))^{-\frac{q}{2}}
|\Omega|^{\frac{2^*_s-q}{2^*_s}}\left(\lambda^{\frac{2}{2-q}}
+\mu^{\frac{2}{2-q}}\right)^{\frac{2-q}{2}}\right]\vspace{0.1cm}\\
&>\left(\frac{2-q}{2(2^*_s-q)}\right)^{\frac{q}{2^*_s-2}}(k_s\S_{s,\alpha,\beta})^{\frac{qN}{4s}}
\Big[\frac{2^*_s-2}{22^*_s}(k_s\S_{s,\alpha,\beta})^{\frac{(2-q)N}{4s}}
\left(\frac{2-q}{2(2^*_s-q)}\right)^{\frac{2-q}{2^*_s-2}}\vspace{0.1cm}\\
&\displaystyle-\frac{2^*_s-q}{q2^*_s}(k_s\S(s,N))^{-\frac{q}{2}}|\Omega|^{\frac{2^*_s-q}{2^*_s}}\left(\lambda^{\frac{2}{2-q}}
+\mu^{\frac{2}{2-q}}\right)^{\frac{2-q}{2}}\Big].
\end{align*}
Thus, if $\lambda^{\frac{2}{2-q}}+\mu^{\frac{2}{2-q}}<(q/2)^{\frac{2}{2-q}}\Lambda_1$,
then
$$
\mathcal {I}_{\lambda,\mu}(z)>c_0,\quad \mbox{for
all $z\in\mathcal{N}^{-}_{\lambda,\mu}$},
$$
for some  positive constant
$c_0=c_0(\lambda,\mu,s,N,|\Omega|).$
\end{proof}

	\bl
	\label{Lemma4.6} 
	Let $(\lambda,\mu)\in {\mathscr
C}_{\Lambda_1}$.\ Then, for every $z=(w_1,w_2)\in
E\backslash\{(0,0)\}$, there
exist unique numbers $t^{-}=t^{-}(z)>0$ and $t^{+}=t^{+}(z)>0$ such that
$$
t^{+}z\in\mathcal{N}^{+}_{\lambda,\mu},\qquad
t^{-}z\in\mathcal{N}^{-}_{\lambda,\mu}.
$$
In particular, we have
$$
t^+<\overline{t}<t^{-}, \qquad
\overline{t}:=\Bigg(\frac{(2^*_{s}-q)Q_{\lambda,\mu}(z)}{(2^*_s-2)\|z\|^2}\Bigg)^{\frac{1}{2-q}}
$$
as well as $t\mapsto \mathcal {I}_{\lambda,\mu}(tz)$ strictly increasing on $[t^+,t^-]$ and 
$$
\mathcal {I}_{\lambda,\mu}(t^{+}z)=\min_{0\leq t\leq t^-}\mathcal {I}_{\lambda,\mu}(tz),\qquad
\mathcal{I}_{\lambda,\mu}(t^{-}z)=\max_{t\geq 0}\mathcal
{I}_{\lambda,\mu}(tz).
$$
\el

\begin{proof}
Fix $z=(w_1,w_2)\in E$, so that
$Q_{\lambda,\mu}(z)>0,$ and let 
$\mathfrak{m}:(0,\infty)\to\R$ be defined by
\be
\label{4.14}
\mathfrak{m}(t):=t^{2-2^*_s}\|z\|^2-t^{q-2^*_s}Q_{\lambda,\mu}(z),~~\text{
for $t>0$}.
\ee Obviously, $\mathfrak{m}(t)\rightarrow-\infty$ as
$t\rightarrow0^+$ and $\mathfrak{m}(t)\to 0$ and $\mathfrak{m}(t)>0$ as
$t\rightarrow\infty.$ Since
$$
\mathfrak{m}'(t)=(2-2^*_s)t^{1-2^*_s}\|z\|^2-(q-2^*_s)t^{q-2^*_s-1}Q_{\lambda,\mu}(z),
$$
we have $\mathfrak{m}'(t)=0$ at $t=\overline{t}=\overline{t}_{\max}$,
$\mathfrak{m}'(t)>0$ for $t\in(0,\overline{t}_{\max})$ and
$\mathfrak{m}'(t)<0$ for $t\in(\overline{t}_{\max},\infty).$ Hence,
$\mathfrak{m}$ achieves its maximum at $\overline{t}_{\max}$, is
increasing for $t\in(0,\overline{t}_{\max})$ and decreasing for
$t\in(\overline{t}_{\max},\infty).$ By $(\lambda,\mu)\in {\mathscr
C}_{\Lambda_1}$, \eqref{4.2} and \eqref{zwm=6}, we have
\begin{align}
\label{4.15}Q_{\lambda,\mu}(z)&\leq
(k_s\S(s,N))^{-\frac{q}{2}}|\Omega|^{\frac{2^*_s-q}{2^*_s}}\left(\lambda^{\frac{2}{2-q}}
+\mu^{\frac{2}{2-q}}\right)^{\frac{2-q}{2}}\|z\|^q  \notag \\
&< (k_s\S(s,N))^{-\frac{q}{2}}|\Omega|^{\frac{2^*_s-q}{2^*_s}}\Lambda_1^{\frac{2-q}{2}}\|z\|^q \\
&=\frac{2^*_s-2}{2^*_s-q}\|z\|^q\left[\frac{2-q}{2(2^*_s-q)}(k_s\S_{s,\alpha,\beta})^{\frac{2^*_s}{2}}\right]^{\frac{2-q}{2^*_s-2}}.
\notag
\end{align}
By \eqref{4.15} and a simple calculation we have
\begin{align}\label{4.16}
\mathfrak{m}(\overline{t}_{\max})&=(\overline{t}_{\max})^{2-2^*_s}\|z\|^2-(\overline{t}_{\max})^{q-2^*_s}Q_{\lambda,\mu}(z)\notag \\
                          &=\left[\frac{(2^*_s-q)Q_{\lambda,\mu}(z)}{(2^*_s-2)\|z\|^2}\right]^{\frac{2-2^*_s}{2-q}}\|z\|^2
                          -\left[\frac{(2^*_s-q)Q_{\lambda,\mu}(z)}{(2^*_s-2)\|z\|^2}\right]^{\frac{q-2^*_s}{2-q}}Q_{\lambda,\mu}(z)\notag \\
                        &=\left(\frac{2^*_s-q}{2^*_s-2}\right)^{\frac{2-2^*_s}{2-q}}\times\frac{2-q}{2^*_s-q}\|z\|^{\frac{2(2^*_s-q)}{2-q}}Q_{\lambda,\mu}(z)^{\frac{2-2^*_s}{2-q}} \\
                       &> 2\|z\|^{2^*_s}(k_s \S_{s,\alpha,\beta})^{-\frac{2^*_s}{2}}.\notag
\end{align}
Then, taking into account the definition of $\S_{s,\alpha,\beta}$ we have
\begin{equation*}
\mathfrak{m}(0)=-\infty<0\leq 2\int_{\Omega}|w_1|^{\alpha}|w_2|^{\beta}dx
\leq 2(k_s\S_{s,\alpha,\beta})^{-\frac{2^*_s}{2}}\|z\|^{2^*_s} 
<\mathfrak{m}(\overline{t}_{\max}).
\end{equation*} 
In turn, there exist unique $t^+$ and $t^-$ such that $0<t^+<\overline{t}_{\max}<t^-,$
$$
\mathfrak{m}(t^+)=2\int_{\Omega}|w_1|^{\alpha}|w_2|^{\beta}dx=\mathfrak{m}(t^-),
$$
and
$\mathfrak{m}'(t^+)>0>\mathfrak{m}'(t^-).$ By the equation  for
$\Phi'_z(t)$ and  \eqref{4.14} we have
\be\label{4.17}
\Phi'_z(t)=t^{2^*_s-1}\left[\mathfrak{m}(t)-2\int_{\Omega}|w_1|^{\alpha}|w_2|^{\beta}dx\right],
\ee
which yields $\Phi_z'(t^\pm)=0$. By virtue of \eqref{4.17},
 we also have $\pm\Phi_z''(t^\pm)>0$. This shows that $\Phi_z$ has a local minimum at $t^+$ and local maximum
 at $t^-$ with $t^\pm z\in\mathcal{N}^{\pm}_{\lambda,\mu}$. The function 
 $t\mapsto \mathcal {I}_{\lambda,\mu}(tz)$ is increasing on $[t^+,t^-]$ and decreasing over $[0,t^+]\cup [t^-,\infty)$. Hence,
 $\mathcal {I}_{\lambda,\mu}(t^{+}z)=\min_{0\leq t\leq t^-}\mathcal {I}_{\lambda,\mu}(tz)$ and $\mathcal
 {I}_{\lambda,\mu}(t^{-}z)=\max_{t\geq 0}\mathcal
 {I}_{\lambda,\mu}(tz)$, concluding the proof.
\end{proof}


\s{Existence of Palais-Smale sequences}


\bl\label{Lemma5.1}
Let $(\lambda,\mu)\in {\mathscr C}_{\Lambda_1}$. Then,
for any $z\in\mathcal{N}_{\lambda,\mu}$, there exists $r>0$ and a differentiable
map $\xi: B(0;r)\subset E\rightarrow\R^+$  such that
$\xi(0)=1$ and $\xi(h)(z-h)\in\mathcal{N}_{\lambda,\mu}$ for every $h\in B(0;r)$.  Let us set
\begin{align*}
\T_1 &:=2k_{s}\int_{\mathcal{C}_\Omega}y^{1-2s}\left(\nabla w_1\cdot \nabla h_1+\nabla w_2\cdot \nabla h_2\right)dxdy, \\
\T_2 &:= q\int_{\Omega}\left(\lambda|w_1|^{q-2}w_1h_1+\mu |w_2|^{q-2}w_2h_2\right)dx, \\
\T_3 & := 2\displaystyle\int_{\Omega}\left(\alpha|w_1|^{\alpha-2}w_1h_1|w_2|^{\beta}
+\beta|w_1|^{\alpha}|w_2|^{\beta-2}w_2h_2\right)dx,
\end{align*}
 for all $(h_1,h_2)\in E$ and $(w_1,w_2)\in E.$ Then
\begin{align}
\label{5.1}
\langle\xi'(0),h\rangle&=\displaystyle\frac{\T_3+\T_2-\T_1}{(2-q)\|z\|^2-2(2^*_s-q)
    \displaystyle\int_{\Omega}|w_1|^{\alpha}|w_2|^{\beta}dx},
\end{align}
     for all $(h_1,h_2)\in E.$
\el
\begin{proof}
For $z=(w_1,w_2)\in\mathcal{N}_{\lambda,\mu}$, define a function ${\mathscr H}_z: \R\times E\rightarrow\R$ by
\begin{align*}
  {\mathscr H}_z(\xi,p) &:=  \langle\mathcal
{I}'_{\lambda,\mu}(\xi(z-p)),\xi(z-p)\rangle\vspace{0.2cm}\\
&=\xi^2k_{s}\int_{\mathcal {C}_\Omega}y^{1-2s}\left(|\nabla
    (w_1-p_1)|^2+|\nabla
    (w_2-p_2)|^2\right)dxdy\\
 & -\xi^q \displaystyle\int_{\Omega}\left(\lambda|w_1-p_1|^{q}+\mu |w_2-p_2|^{q}\right)dx-
    2\xi^{2^*_s}\int_{\Omega}|w_1-p_1|^{\alpha}|w_2-p_2|^{\beta}dx.
    \end{align*}
    Then ${\mathscr H}_z(1,0)=\langle\mathcal
{I}'_{\lambda,\mu}(z),z\rangle=0$ and, by Lemma~\ref{Lemma4.3}, we have
\begin{align*}
 \frac{d {\mathscr H}_z(1,(0,0))}{d\xi}&=\displaystyle2\|z\|^2-q\int_{\Omega}\left(\lambda|w_1|^{q}+\mu |w_2|^{q}\right)dx-
    22^*_s\int_{\Omega}|w_1|^{\alpha}|w_2|^{\beta}dx \\
    &= \displaystyle(2-q)\|z\|^2-2(2^*_s-q)\int_{\Omega}|w_1|^{\alpha}|w_2|^{\beta}dx\not=0.
 \end{align*}
 In turn, by virtue of the Implicit Function Theorem, there
 exists $r>0$ and a function $\xi: B(0;r)\subset E\to \R$
 of class $C^1$ such that $\xi(0)=1$ and formula \eqref{5.1} holds, via direct computation.
Moreover, ${\mathscr H}_z(\xi(h),h)=0$, for all $h\in B(0;r),$ is equivalent to
    $$\langle\mathcal
{I}'_{\lambda,\mu}(\xi(h)(z-h)),\xi(h)(z-h)\rangle=0,
\quad\text{for all $h\in B(0;r)$},
$$
namely $\xi(h)(z-h)\in\mathcal{N}_{\lambda,\mu}.$
\end{proof}

\bl\label{Lemma5.2}
Let $(\lambda,\mu)\in {\mathscr C}_{\Lambda_1}$.
Then, for each $z\in\mathcal{N}^-_{\lambda,\mu}$, there exists $r>0$ and a differentiable function $\xi^-: B(0;r)\subset E\rightarrow\R^+$  such that
$\xi^-(0)=1,$  $\xi^-(h)(z-h)\in\mathcal{N}^{-}_{\lambda,\mu}$  for every $h\in B(0;r)$
and formula \eqref{5.1} holds.
\el

\begin{proof}
Arguing as for the proof of Lemma \ref{Lemma5.1}, there exists
$r>0$ and a differentiable function $\xi^-: B(0;r)\subset
E\rightarrow\R^+$ such that $\xi^-(0)=1,$
$\xi^-(h)(z-h)\in\mathcal{N}_{\lambda,\mu}$ for all $h\in B(0;r)$
and formula \eqref{5.1} holds. Since
$$
\langle\mathcal{R}'_{\lambda,\mu}(z),z\rangle=(2-q)\|z\|^2-2(2^*_s-q)\int_{\Omega}|w_1|^{\alpha}|w_2|^{\beta}dx<0,
$$
by the continuity of the functions $\mathcal{R}'_{\lambda,\mu}$ and $\xi^-$, up to reducing
the size of $r>0$, we get
\begin{align*}
\langle\mathcal{R}'_{\lambda,\mu}&(\xi^-(h)(z-h)),\xi^-(h)(z-h)\rangle
=\displaystyle(2-q)\|\xi^-(h)(z-h)\|^2  \\
&-2(2^*_s-q)
\displaystyle\int_{\Omega}|(\xi^-(h)(z-h))_1|^{\alpha}|(\xi^-(h)(z-h))_2|^{\beta}dx<0,
\end{align*}
where $(\xi^-(h)(z-h))_i\in X^{s}_0(\mathcal {C}_{\Omega})$ denote
the components of $\xi^-(h)(z-h)$. This implies that the functions $\xi^-(h)(z-h)$
belong to $\mathcal{N}^{-}_{\lambda,\mu}.$
\end{proof}

\begin{proposition}
    \label{Proposition5.3}
The following facts hold.
\begin{itemize}
\item[$(i)$] Let $(\lambda,\mu)\in {\mathscr C}_{\Lambda_1}.$
Then there is a $(PS)_{\alpha_{\lambda,\mu}}$-sequence
$\{z_n\}\subset\mathcal{N}_{\lambda,\mu}$ for $\mathcal{I}_{\lambda,\mu}$.
 \item[$(ii$)] Let $(\lambda,\mu)\in {\mathscr C}_{(q/2)^{2/(2-q)}\Lambda_1}$.\ Then
there is a $(PS)_{\alpha_{\lambda,\mu}^{-}}$-sequence
$\{z_n\}\subset\mathcal{N}^-_{\lambda,\mu}$ for $\mathcal{I}_{\lambda,\mu}$.
\end{itemize}
\end{proposition}

\begin{proof}
(i) By Lemma  \ref{Lemma4.1} and Ekeland Variational Principle \cite{E}, there exists a minimizing sequence $\{z_n\}\subset\mathcal{N}_{\lambda,\mu}$ such that
\begin{align}
\mathcal{I}_{\lambda,\mu}(z_n) & <\alpha_{\lambda,\mu}+\frac{1}{n},   \notag \\
\label{5.3}
\mathcal{I}_{\lambda,\mu}(z_n)&<\mathcal
{I}_{\lambda,\mu}(w)+\frac{1}{n}\|w-z_n\|,\quad\,\, \mbox{for each $w\in\mathcal{N}_{\lambda,\mu}$}.
\end{align}
Taking $n$ large and using $\alpha_{\lambda,\mu}<0$, we have
\begin{align}
\label{5.4}
\mathcal{I}_{\lambda,\mu}(z_n)&=\displaystyle\Big(\frac{1}{2}-\frac{1}{2^*_s}\Big)\|z_n\|^2
-\Big(\frac{1}{q}-\frac{1}{2^*_s}\Big)\int_{\Omega}(\lambda|w_{1,n}|^q+\mu|w_{2,n}|^q)dx\vspace{0.1cm}\\
&<\displaystyle\alpha_{\lambda,\mu}+\frac{1}{n}<\frac{\alpha_{\lambda,\mu}}{2}.  \notag
\end{align}
This yields that
\begin{align}
\label{5.5}
-\frac{q2^*_s}{2(2^*_s-q)}\alpha_{\lambda,\mu}&<\int_{\Omega}(\lambda|w_{1,n}|^q+\mu|w_{2,n}|^q)dx  \\
&\leq\displaystyle |\Omega|^{\frac{2^*_s-q}{2^*_s}}(k_s\S(s,N))^{-\frac{q}{2}}(\lambda^{\frac{2}{2-q}}+\mu^{\frac{2}{2-q}})^{\frac{2-q}{2}}
\|z_n\|^q.  \notag
\end{align}
Consequently, $z_n\not=0$ and combining with
 \eqref{5.4} and \eqref{5.5} and using H\"{o}lder inequality
\begin{equation*}
\|z_n\|>\left[-\frac{q2^*_s}{2(2^*_s-q)}\alpha_{\lambda,\mu}|\Omega|^{\frac{q-2^*_s}{2^*_s}}
(k_s\S(s,N))^{\frac{q}{2}}(\lambda^{\frac{2}{2-q}}+\mu^{\frac{2}{2-q}})^{\frac{q-2}{2}}\right]^{\frac{1}{q}},
\end{equation*}
and
\begin{equation}
\label{5.7}
\|z_n\|<\left[\frac{2(2^*_s-q)}{q(2^*_s-2)}|\Omega|^{\frac{2^*_s-q}{2^*_s}}
(k_s\S(s,N))^{-\frac{q}{2}}(\lambda^{\frac{2}{2-q}}+\mu^{\frac{2}{2-q}})^{\frac{2-q}{2}}\right]^{\frac{1}{2-q}}.
\end{equation}
Now we prove that
$$\|\mathcal{I}'_{\lambda,\mu}(z_n)\|_{E^{-1}}\rightarrow0,~~\mbox{as}~n\rightarrow\infty.$$
Fix $n\in{\mathbb N}$. By applying Lemma  \ref{Lemma5.1} to $z_n$, we obtain the function $\xi_n:
B(0;r_n)\rightarrow\R^+$ for some $r_n>0,$ such that
$\xi_n(h)(z_n-h)\in\mathcal{N}_{\lambda,\mu}.$ Take $0<\rho<r_n.$
Let $w\in E$ with $w\not\equiv0$ and put $h^*=\frac{\rho w}{\|w\|}.$
We set  $h_{\rho}=\xi_n(h^*)(z_n-h^*)$, then
$h_{\rho}\in\mathcal{N}_{\lambda,\mu},$ and we have from (\ref{5.3}) that
$$\mathcal
{I}_{\lambda,\mu}(h_\rho)-\mathcal
{I}_{\lambda,\mu}(z_n)\geq-\frac{1}{n}\|h_\rho-z_n\|.$$
By the Mean Value Theorem, we get
$$\langle\mathcal
{I}'_{\lambda,\mu}(z_n),h_\rho-z_n\rangle+o(\|h_\rho-z_n\|)\geq-\frac{1}{n}\|h_\rho-z_n\|.$$
Thus, we have
\begin{equation*}
\langle\mathcal
{I}'_{\lambda,\mu}(z_n),-h^*\rangle+(\xi_n(h^*)-1)\langle\mathcal
{I}'_{\lambda,\mu}(z_n),z_n-h^*\rangle\geq-\frac{1}{n}\|h_\rho-z_n\|+o(\|h_\rho-z_n\|).
\end{equation*}
Whence, from $\xi_n(h^*)(z_n-h^*)\in\mathcal{N}_{\lambda,\mu},$ it follows that
\begin{align*}
    -\rho\Big\langle\mathcal
{I}'_{\lambda,\mu}(z_n),\frac{w}{\|w\|}\Big\rangle&+(\xi_n(h^*)-1)\langle\mathcal
{I}'_{\lambda,\mu}(z_n)-\mathcal
{I}'_{\lambda,\mu}(h_\rho),z_n-h^*\rangle  \\
&\geq-\frac{1}{n}\|h_\rho-z_n\|+o(\|h_\rho-z_n\|).
\end{align*}
So, we get
\begin{align}
\label{5.10}
\displaystyle\Big\langle\mathcal
{I}'_{\lambda,\mu}(z_n),\frac{w}{\|w\|}\Big\rangle&\leq\displaystyle\frac{1}{n\rho}\|h_\rho-z_n\|
+\frac{o(\|h_\rho-z_n\|)}{\rho}\vspace{0.2cm}\\
&+\displaystyle\frac{(\xi_n(h^*)-1)}{\rho}\langle\mathcal
{I}'_{\lambda,\mu}(z_n)-\mathcal
{I}'_{\lambda,\mu}(h_\rho),z_n-h^*\rangle.  \notag
\end{align}
Since $\|h_\rho-z_n\|\leq\rho|\xi_n(h^*)|+|\xi_n(h^*)-1|\|z_n\|$ and
$$
\lim_{\rho\rightarrow0}\frac{|\xi_n(h^*)-1|}{\rho}\leq\|\xi_n'(0)\|.
$$
For  fixed $n\in{\mathbb N}$, if we let $\rho\rightarrow0$ in \eqref{5.10}, then by virtue of
\eqref{5.7} we can choose a constant $C>0$ independent of $\rho$ such that
\begin{equation*}
\Big\langle
\mathcal{I}'_{\lambda,\mu}(z_n),\frac{w}{\|w\|}\Big\rangle\leq\frac{C}{n}(1+\|\xi_n'(0)\|).
\end{equation*}
Thus, we are done once we prove that $\|\xi_n'(0)\|$ remains uniformly bounded.
By \eqref{5.1}, \eqref{5.7} and  H\"{o}lder inequality, we have
$$
\big|\langle\xi_n'(0),h\rangle\big|\leq\frac{C_1\|h\|}{\Big|(2-q)\|z_n\|^2-2(2^*_s-q)
    \displaystyle\int_{\Omega}|w_{1,n}|^{\alpha}|w_{2,n}|^{\beta}dx\big|}
$$
    for some $C_1>0.$ We only need to prove that
    $$
    \left|{(2-q)\|z_n\|^2-2(2^*_s-q)
    \int_{\Omega}|w_{1,n}|^{\alpha}|w_{2,n}|^{\beta}dx}\right|\geq C_2,
$$
    for some $C_2>0$ and $n$ large enough. We argue by contradiction. Suppose that there exists a subsequence $\{z_n\}$ such that
\be
\label{5.12}
(2-q)\|z_n\|^2-2(2^*_s-q)
    \int_{\Omega}|w_{1,n}|^{\alpha}|w_{2,n}|^{\beta}dx=o_n(1).
    \ee
By virtue of \eqref{5.12} and the fact that $z_n\in\mathcal{N}_{\lambda,\mu}$, we have
$$
\|z_n\|^2=\frac{2(2^*_s-q)}{2-q}\int_{\Omega}|w_{1,n}|^{\alpha}|w_{2,n}|^{\beta}dx+o_n(1),\quad\,\,
\|z_n\|^2=\frac{2^*_s-q}{2^*_s-2}Q_{\lambda,\mu}(z_n)+o_n(1).
$$
Taking into account that $\mathcal{I}_{\lambda,\mu}(z_n)\to\alpha_{\lambda,\mu}<0$ as $n\to\infty,$ we have $\|z_n\|\not\to 0$ as $n\to\infty$.
Then, arguing as in the proof of Lemma~\ref{Lemma4.3} yields
$(\lambda,\mu)\not\in {\mathscr C}_{\Lambda_1}$, a contradiction. Then,
$$
\Big\langle\mathcal{I}'_{\lambda,\mu}(z_n),\frac{w}{\|w\|}\big\rangle\leq\frac{C}{n}.
$$
This proves (i).\ By Lemma \ref{Lemma5.2}, one can prove (ii),
but we shall omit the details here.
\end{proof}

\section{Local minimization problems}

Now, we establish the existence of a local minimizer for  $\mathcal
{I}_{\lambda,\mu}$ in $\mathcal{N}^{+}_{\lambda,\mu}$.

\begin{proposition}\label{Proposition5.4}
    Let $(\lambda,\mu)\in {\mathscr C}_{\Lambda_1}$. Then
$\mathcal {I}_{\lambda,\mu}$ has  a local minimizer $z^+$    in
$\mathcal{N}^{+}_{\lambda,\mu}$ satisfying the following
conditions:
\begin{itemize}
\item[(i)] $\mathcal
{I}_{\lambda,\mu}(z^+)=\alpha_{\lambda,\mu}=\alpha^+_{\lambda,\mu}<0;$
\item[(ii)] $z^+$ is a positive solution of \eqref{2.6}.
\end{itemize}
\end{proposition}

\begin{proof}
    By (i) of Proposition \ref{Proposition5.3}, there exists a minimizing
sequence $\{z_n\}=\{(w_{1,n},w_{2,n})\}$ for
$\mathcal {I}_{\lambda,\mu}$ in $\mathcal{N}_{\lambda,\mu}$ such that, as $n\to\infty,$
\be\label{5.13}
\mathcal {I}_{\lambda,\mu}(z_n)=\alpha_{\lambda,\mu}+o_n(1)
\quad\mbox{and}\quad
\mathcal{I}'_{\lambda,\mu}(z_n)=o_n(1)~~\mbox{in}~E^{-1}.
\ee
By Lemma \ref{Lemma4.1},  we see that
$\mathcal {I}_{\lambda,\mu}$ is coercive on $
\mathcal{N}_{\lambda,\mu}$, and $\{z_n\}$ is bounded in $E$.
Then there exists a subsequence, still denoted by $\{z_n\}$ and
$z^+=(w_1^+,w_2^+)\in E$ such that, as $n\to\infty$,
\begin{equation*}\left\{\begin{array}{l}
\displaystyle w_{1,n}\rightharpoonup w^+_1, \,\,\, w_{2,n}\rightharpoonup
w_2^+, \,\,\,\mbox{weakly
in}~X^s_0(\Omega),  \\
\displaystyle w_{1,n}\rightarrow w_1^+, \,\,\, w_{2,n}\rightarrow w_2^+, \,\,\,
\mbox{strongy in $L^r(\Omega)$ for all $1\leq r<2^*_s$},   \\
\displaystyle w_{1,n}\rightarrow w_1^+, \,\,\, w_{2,n}\rightarrow w_2^+, \,\,\,
\mbox{a.e.\ in}~\Omega,
\end{array}
\right.
\end{equation*}
up to subsequences. This implies that, as $n\to\infty,$
\be
\label{5.15}
Q_{\lambda,\mu}(z_n)=Q_{\lambda,\mu}(z^+)+o_n(1).
\ee
We  claim that $z^+$ is a nontrivial solution of $(\ref{2.6})$. It is easy to verify that $z^+$ is a
weak solution of $(\ref{2.6})$. From $z_n\in\mathcal{N}_{\lambda,\mu}$ and
(\ref{4.2}) we deduce that
\be\label{5.16}Q_{\lambda,\mu}(z_n)=\frac{q(2^*_s-2)}{2(2^*_s-q)}\|z_n\|^2-\frac{q2^*_s}{2^*_s-q}\mathcal
{I}_{\lambda,\mu}(z_n).\ee
 Let $n\rightarrow\infty$ in
\eqref{5.16}, by \eqref{5.13}, \eqref{5.15} and $\alpha_{\lambda,\mu}<0,$ we have
$$Q_{\lambda,\mu}(z^+)\geq-\frac{q2^*_s}{2^*_s-q}\alpha_{\lambda,\mu}>0.$$
Therefore, $z^+\in\mathcal{N}_{\lambda,\mu}$ is a nontrivial
solution of $(\ref{2.6})$. Now we show that $z_n\rightarrow z^+$
strongly in $E$ and $\mathcal
{I}_{\lambda,\mu}(z^+)=\alpha_{\lambda,\mu}$. Since
$z^+\in\mathcal{N}_{\lambda,\mu}$, then by \eqref{5.16}, we obtain
\begin{align*}
\displaystyle \alpha_{\lambda,\mu}&\leq\mathcal {I}_{\lambda,\mu}(z^+)\vspace{0.1cm}\\
                      &=\displaystyle \frac{s}{N}\|z^+\|^2-\frac{2^*_s-q}{q2^*_s}Q_{\lambda,\mu}(z^+)\vspace{0.1cm}\\
                      &\leq\displaystyle \lim_{n\rightarrow\infty}\Big(\frac{s}{N}\|z_n\|^2-\frac{2^*_s-q}{q2^*_s}Q_{\lambda,\mu}(z_n)\Big)\vspace{0.1cm}\\
                      &=\displaystyle \lim_{n\rightarrow\infty}\mathcal
{I}_{\lambda,\mu}(z_n)=\alpha_{\lambda,\mu}.
\end{align*}
This implies that $\mathcal
{I}_{\lambda,\mu}(z^+)=\alpha_{\lambda,\mu}$ and
$\lim_{n\rightarrow\infty}\|z_n\|^2=\|z^+\|^2.$ Since
 $$
 \|z_n-z^+\|^2=\|z_n\|^2-\|z^+\|^2+o_n(1),
 $$
we conclude that $z_n\rightarrow z^+$ in $E$. We claim that
$z^+\in\mathcal{N}^+_{\lambda,\mu}$. Assume by contradiction that
$z^+\in\mathcal{N}^{-}_{\lambda,\mu}$. Then, by Lemma
\ref{Lemma4.6}, there exist (unique) $t^+_1$ and $t^{-}_1$ with
$t^+_1z^+\in\mathcal{N}^+_{\lambda,\mu}$ and
$t^{-}_1z^+\in\mathcal{N}^{-}_{\lambda,\mu}$. In particular, we have
$t^+_1<t^{-}_1=1.$ Since
$$
\frac{d}{dt}\mathcal
{I}_{\lambda,\mu}(tz^+)|_{t=t^+_1}=0,
\quad\mbox{and}\quad
\frac{d^2}{dt^2}\mathcal
{I}_{\lambda,\mu}(tz^+)|_{t=t^+_1}>0,
$$ there exists $t^+_1<t^*\leq
t^{-}_1$ such that $\mathcal {I}_{\lambda,\mu}(t^+_1z^+)<\mathcal
{I}_{\lambda,\mu}(t^*z^+)$. By Lemma \ref{Lemma4.6}, we have
$$
\alpha_{\lambda,\mu}\leq \mathcal {I}_{\lambda,\mu}(t^+_1z^+)<\mathcal
{I}_{\lambda,\mu}(t^*z^+)\leq\mathcal
{I}_{\lambda,\mu}(t^{-}_1z^+)=\mathcal {I}_{\lambda,\mu}(z^+)=\alpha_{\lambda,\mu},
$$
a contradiction. Since $\mathcal{I}_{\lambda,\mu}(z^+)=\mathcal{I}_{\lambda,\mu}(|w_1^+|,|w_2^+|)$
 and $(|w_1^+|,|w_2^+|)\in\mathcal{N}_{\lambda,\mu}$, by Lemma \ref{Lemma4.2} we may assume that $z^+$ is a nontrivial nonnegative solution of \eqref{2.6}.
In particular $w^+_1\not\equiv0,w_2^+\not\equiv0.$ In fact, without loss of
generality, we assume by contradiction that $w_2^+\equiv0.$ Then $w_1^+$ is a
nontrivial nonnegative solution of
\begin{equation*}
\begin{cases}
            -\mbox{div}(y^{1-2s}\nabla w)=0 &\mbox{in}~ \mathcal
{C}_{\Omega}\vspace{0.1cm}\\
           w=0 &\mbox{on}~ \partial_L\mathcal
{C}_{\Omega}\vspace{0.1cm}\\

w=u&\mbox{on}~{\Omega}\times\{0\}\vspace{0.1cm}\\
            \frac{\partial w}{\partial\nu^s}=\lambda|w|^{q-2}w&\mbox{on}
            ~\Omega\times\{0\},
 \end{cases}
\end{equation*}By the maximum principle \cite{SV7} we get $w_1^+>0$
in $X^{s}_0(\mathcal {C}_{\Omega})$ and
$$\|(w_1^+,0)\|^2=Q_{\lambda,\mu}(w_1^+,0)>0.$$Moreover, we can
choose $w^*_2\in X^s_0(\mathcal {C}_{\Omega})\backslash\{0\}$ such
that
$$\|(0,w^*_2)\|^2=Q_{\lambda,\mu}(0,w^*_2)>0.$$Note that
$$Q_{\lambda,\mu}(w_1^+,w^*_2)=Q_{\lambda,\mu}(w_1^+,0)+Q_{\lambda,\mu}(0,w^*_2)>0,$$and
so by Lemma \ref{Lemma4.6} there is unique
$0<t^+<\overline{t}<t^-$ such that $(t^+w_1^+,t^+w^*_2)\in
\mathcal {N}_{\lambda,\mu}^+.$ Moreover,
$$
\overline{t}=\left(\frac{(2^*_{s}-q)Q_{\lambda,\mu}(w_1^+,w^*_2)}{(2^*_s-2)
	\|(w_1^+,w^*_2)\|^2}\right)^{\frac{1}{2-q}}=\left(\frac{2^*_{s}-q}{2^*_s-2}\right)^{\frac{1}{2-q}}>1
$$
and
$$\mathcal{I}_{\lambda,\mu}(t^+w_1^+,t^+w^*_2)=\inf_{0\leq
t\leq t^-}\mathcal
{I}_{\lambda,\mu}(tw_1^+,tw^*_2).$$This implies that
$$\alpha^+_{\lambda,\mu}\leq\mathcal
{I}_{\lambda,\mu}(t^+w_1^+,t^+w^*_2)\leq \mathcal
{I}_{\lambda,\mu}(w_1^+,w^*_2)<\mathcal
{I}_{\lambda,\mu}(w_1^+,0)=\alpha^+_{\lambda,\mu}$$which is a
contradiction. Then by the Strong Maximum Principle \cite[Lemma 2.4]{capella},
 we have $w^+_1,w_2^+>0$ in $\mathcal{C}_{\Omega},$ hence, $z^+$ is a positive solution for \eqref{2.6}.
\end{proof}

\noindent
Next we will use $w_{\varepsilon}=E_{s}(u_{\varepsilon})$, the family of minimizers for the trace inequality \eqref{2.10},
where $u_{\varepsilon}$ is given in \eqref{2.11}. Without loss of generality, we may assume that $0\in\Omega.$
We then define  the cut-off function  $\phi\in C^{\infty}_0(\mathcal{C}_{\Omega}), 0\leq\phi\leq1$ and for small fixed $\rho>0$,
 $$\phi(x,y)=\left\{\begin{array}{ll}
1,&(x,y)\in B_{\rho},\\
0,&(x,y)\not\in \overline{B_{2\rho}},
\end{array}\right.$$
where
$B_{\rho}=\{(x,y):|x|^2+y^2<\rho^2, y>0\}.$ We take $\rho$ so small
that $\overline{B_{2\rho}}\subset\overline{\mathcal{C}_{\Omega}}.$ Recall
$\mathcal{W}$ is the extension of $U$ introduced in Section 2, we
have (cf.\ \cite{B1}) $|\nabla\mathcal{W}(x,y)|\leq
Cy^{-1}\mathcal{W}(x,y)$. Let
$$
U_{\varepsilon}(x)=\frac{1}{(\varepsilon^2+|x|^2)^{\frac{N-2s}{2}}},
\quad \varepsilon>0.
$$
 Then the
extension of $U_{\varepsilon}(x)$ has the form
$$\mathcal{W}_{\varepsilon}(x,y)=c_{N,s}y^{2s}\int_{\R^N}\frac{U_{\varepsilon}(z)dz}{(|x-z|^2+y^2)^{\frac{N+2s}{2}}}
=\varepsilon^{2s-N}\mathcal{W}\left(\frac{x}{\varepsilon},\frac{y}{\varepsilon}\right).$$
Notice that $\phi \mathcal{W}_{\varepsilon}\in
X^s_0(\mathcal{C}_{\Omega})$, for $\varepsilon>0$ small enough.

\bl\label{Lemma5.5} There is
$z\in E\backslash\{0\}$ nonnegative
and ${\Lambda^*}>0$ such that for $(\lambda,\mu)\in {\mathscr C}_{\Lambda^*}$
$$\sup_{t\geq 0}\mathcal
{I}_{\lambda,\mu}(tz)<c_{\infty},$$ where $c_{\infty}$ is given in
Lemma~\ref{Lemma3.3}. In particular, $\alpha^{-}_{\lambda,\mu}<c_{\infty}$ for all
$(\lambda,\mu)\in {\mathscr C}_{\Lambda^*}$.
\el
\begin{proof}
By an argument similar to that of
the proof of \cite[formula (3.26)]{B1}, we get
\begin{align}
\label{5.18}
\displaystyle\|\phi\mathcal{W}_{\varepsilon}\|^2_{X^s_0}&=\displaystyle k_s\int_{\R^{N+1}_{+}}y^{1-2s}|\nabla\mathcal{W}_{\varepsilon}|^2dxdy
+\O(1)\vspace{0.2cm}\\
&=\displaystyle\varepsilon^{2s-N}k_s\int_{\R^{N+1}_{+}}y^{1-2s}|\nabla\mathcal{W}(x,y)|^2dxdy
+\O(1).  \notag
\end{align}
We notice that
\begin{align*}
\|\phi U_{\varepsilon}\|^{2^*_s}_{2^*_s}&=\displaystyle\int_{\Omega}|\phi
U_{\varepsilon}|^{2^*_s}d
x=\int_{\Omega}\frac{\phi(x)^{2^*_s}}{(\varepsilon^2+|x|^2)^N}dx, \\
\|U_{\varepsilon}\|^{2^*_s}_{2^*_s}&=\int_{\R^N}\frac{1}{(\varepsilon^2+|x|^2)^N}dx
=\varepsilon^{-N}\|U\|^{2^*_s}_{2^*_s}.
\end{align*}
Then, one has that
$$\|\phi U_{\varepsilon}\|^{2^*_s}_{2^*_s}-\varepsilon^{-N}\|U\|^{2^*_s}_{2^*_s}=
\int_{\Omega}\frac{\phi^{2^*_s}(x)-1}{(\varepsilon^2+|x|^2)^N}dx-\int_{\R^N\backslash\Omega}\frac{dx}
{(\varepsilon^2+|x|^2)^N},
$$
which yields
\begin{align*}
\displaystyle\left|\|\phi U_{\varepsilon}\|^{2^*_s}_{2^*_s}-\varepsilon^{-N}\|U\|^{2^*_s}_{2^*_s}\right|&\leq\displaystyle\int_{\Omega\backslash B(0;\rho)}\frac{1}{(\varepsilon^2+|x|^2)^N}dx+\int_{\R^N\backslash\Omega}\frac{dx}
{(\varepsilon^2+|x|^2)^N}\vspace{0.2cm}\\
&=\displaystyle\int_{\R^N\backslash B(0;\rho)}\frac{dx}{(\varepsilon^2+|x|^2)^N}
\leq\displaystyle\int_{\R^N\backslash B(0;\rho)}\frac{dx}{|x|^{2N}}=C_3.
\end{align*}
This implies that
$$
1-C_3\varepsilon^{N}\|U\|^{-2^*_s}_{2^*_s}\leq\varepsilon^{N}\|\phi U_\varepsilon\|^{2^*_s}_{2^*_s}\|U\|^{-2^*_s}_{2^*_s}\leq1+C_3\varepsilon^{N}\|U\|^{-2^*_s}_{2^*_s}.
$$
Taking $\varepsilon$ so small that
$C_3\varepsilon^{N}\|U\|^{-2^*_s}_{2^*_s}<1$, since $2/2^*_s=(N-2s)/N<1$, we obtain
\begin{align*}
\displaystyle 1-\varepsilon^{N}C_3\|U\|^{-2^*_s}_{2^*_s}&\leq\displaystyle (1-\varepsilon^{N}C_3\|U\|^{-2^*_s}_{2^*_s})^{2/2^*_s}
\leq\displaystyle \varepsilon^{N-2s}\|\phi U_{\varepsilon}\|^{2}_{2^*_s}\|U\|^{-2}_{2^*_s}   \\
&\leq\displaystyle (1+\varepsilon^{N}C_3\|U\|^{-2^*_s}_{2^*_s})^{2/2^*_s}
\leq\displaystyle 1+\varepsilon^{N}C_3\|U\|^{-2^*_s}_{2^*_s}.
\end{align*}
Hence
$\|\phi U_{\varepsilon}\|^{2}_{2^*_s}=\varepsilon^{2s-N}\|U\|^{2}_{2^*_s}+\O(\varepsilon^{2s})$.
Since $\mathcal{W}=E_s(U)$ optimizes \eqref{2.10},  by \eqref{5.18} we have
\begin{align}\label{5.20}
\displaystyle \frac{\|\phi\mathcal{W}_{\varepsilon}\|^2_{X^s_0}}{\|\phi U_{\varepsilon}\|^{2}_{2^*_s}}&=\displaystyle \frac{\varepsilon^{2s-N}k_s \displaystyle\int_{\R^{N+1}_{+}}y^{1-2s}|\nabla\mathcal{W}(x,y)|^2dxdy
+\O(1)}{\varepsilon^{2s-N}\|U\|^{2}_{2^*_s}+\O(\varepsilon^{2s})}\vspace{0.2cm}\\
&=\displaystyle \frac{k_s \displaystyle\int_{\R^{N+1}_{+}}y^{1-2s}|\nabla\mathcal{W}(x,y)|^2dxdy}{\|U\|^{2}_{2^*_s}}
\Big(1+\O(\varepsilon^{N-2s})\Big)  \notag \\
&=\displaystyle k_s\S(s,N)+\O(\varepsilon^{N-2s}).   \notag
\end{align}
Now we consider the function $J:E\to\R$ defined by
$J(z):={1/2}\|z\|^2-{2/2^*_s}\int_{\Omega}|w_{1}|^{\alpha}|w_{2}|^{\beta}dx$. Set
$w_{0,1}:=\sqrt{\alpha}\phi\mathcal{W}_{\varepsilon}$,
$w_{0,2}:=\sqrt{\beta}\phi\mathcal{W}_{\varepsilon}$ and
$z_0:=(w_{0,1},w_{0,2})\in E.$ Notice that $J(0)=0$, $J(tz_0)>0$ for
$t>0$ small and $J(tz_0)<0$ for $t>0$ large.
The map $t\mapsto J(tz_0)$ maximizes at 
\be \label{5.22}
t_0:=\left(\frac{\|z_0\|^2}{2\displaystyle\int_{\Omega}|w_{0,1}|^{\alpha}|w_{0,2}|^{\beta}dx}\right)^{\frac{1}{2^*_s-2}}.
\ee 
Then from \eqref{2.13}, \eqref{5.20} and (\ref{5.22}), we
conclude that
\begin{align}
\label{5.23}
\sup_{t\geq 0}J(tz_0)&=J(t_0z_0)
=\displaystyle\left(\frac{1}{2}-\frac{1}{2^*_s}\right)\frac{\|z_0\|^{\frac{22^*_s}{2^*_s-2}}}
{\Big(2\displaystyle\int_{\Omega}|w_{0,1}|^{\alpha}|w_{0,2}|^{\beta}\Big)^{\frac{2}{2^*_s-2}}}   \notag\\
&=\displaystyle\frac{s}{N}\Bigg[\frac{(\alpha+\beta)k_s\displaystyle\int_{\mathcal{C}_\Omega}y^{1-2s}|\nabla (\phi \mathcal{W}_{\varepsilon})|^2dxdy}
{\Big(\alpha^{\frac{\alpha}{2}}\beta^{\frac{\beta}{2}}\displaystyle\int_{\Omega}
|\phi{U}_{\varepsilon}|^{2^*_s}dx\Big)^{\frac{2}{2^*_s}}}\Bigg]^{\frac{2^*_s}{2^*_s-2}}\cdot\frac{1}{2^{\frac{N-2s}{2s}}}    \notag\\
&=\displaystyle\frac{s}{N2^{\frac{N-2s}{2s}}}\Big[\left(\frac{\alpha}{\beta}\right)^{\frac{\beta}{\alpha+\beta}}+
\left(\frac{\beta}{\alpha}\right)^{\frac{\alpha}{\alpha+\beta}}\Big]^{\frac{N}{2s}}
 \Bigg[\frac{k_s\displaystyle\int_{\mathcal{C}_\Omega}y^{1-2s}|\nabla(\phi \mathcal{W}_{\varepsilon})|^{2}}{\Big(\displaystyle\int_{\Omega}
|\phi{U}_{\varepsilon}|^{2^*_s}dx\Big)^{\frac{2}{2^*_s}}}\Bigg]^{\frac{N}{2s}}   \\
&=\displaystyle\frac{s}{N2^{\frac{N-2s}{2s}}}\Big[\left(\frac{\alpha}{\beta}\right)^{\frac{\beta}{\alpha+\beta}}
+\left(\frac{\beta}{\alpha}\right)^{\frac{\alpha}{\alpha+\beta}}\Big]^{\frac{N}{2s}}
 \Big[k_s\S(s,N)+\O(\varepsilon^{N-2s})\Big]^{\frac{N}{2s}}    \notag \\
&=\displaystyle\frac{s}{N2^{\frac{N-2s}{2s}}}\Big[k_s\S_{s,\alpha,\beta}+\O(\varepsilon^{N-2s})\Big]^{\frac{N}{2s}}
 \notag\\
&=\displaystyle\frac{s}{N}\frac{1}{2^{\frac{N-2s}{2s}}}(k_s\S_{s,\alpha,\beta})^{\frac{N}{2s}}+\O(\varepsilon^{N-2s})
 =\displaystyle\frac{2s}{N}\left(\frac{k_s\S_{s,\alpha,\beta}}{2}\right)^{\frac{N}{2s}}+\O(\varepsilon^{N-2s}).  \notag
\end{align}
 We now choose $\delta_1>0$ so small that, for all
$(\lambda,\mu)\in {\mathscr C}_{\delta_1}$, we get
$$
c_\infty=\frac{2s}{N}\left(\frac{k_s\S_{s,\alpha,\beta}}{2}\right)^{\frac{N}{2s}}
-K_0\big(\lambda^{\frac{2}{2-q}}+\mu^{\frac{2}{2-q}}\big)>0.
$$
By the definition of $\mathcal {I}_{\lambda,\mu}$ and $z_0$, we have
$$
\mathcal{I}_{\lambda,\mu}(tz_0)\leq\frac{t^2}{2}\|z_0\|^2,
\quad
\mbox{for all $t\geq 0$ and $\lambda,\mu>0$},
$$
which implies that there exists $t_0\in(0,1)$ satisfying
$$\sup_{t\in [0,t_0]}\mathcal
{I}_{\lambda,\mu}(tz_0)<c_{\infty},\quad\mbox{for
all $(\lambda,\mu)\in {\mathscr C}_{\delta_1}$}.
$$
Hence, from \eqref{5.23} and $\alpha,\beta>1$ we see that
\begin{align}
\label{5.24}
\sup_{t\geq t_0}\mathcal
{I}_{\lambda,\mu}(tz_0)   
&=\displaystyle\sup_{t\geq t_0}\Big(J(tz_0)-\frac{t^q}{q}Q_{\lambda,\mu}(z_0)\Big)\vspace{0.1cm}\\
&\leq\displaystyle\frac{2s}{N}\Big(\frac{k_s\S_{s,\alpha,\beta}}{2}\Big)^{\frac{N}{2s}}+\O(\varepsilon^{N-2s})
\displaystyle-\frac{t^q_0}{q}\left(\lambda\alpha^{\frac{q}{2}}+
\mu\beta^{\frac{q}{2}}\right)\int_{B(0;\rho)}|{U}_{\varepsilon}|^qdx  \notag \\
&\leq \displaystyle\frac{2s}{N}\Big(\frac{k_s\S_{s,\alpha,\beta}}{2}\Big)^{\frac{N}{2s}}+\O(\varepsilon^{N-2s})
\displaystyle-\frac{t^q_0}{q}\left(\lambda+
\mu\right)\int_{B(0;\rho)}|{U}_{\varepsilon}|^qdx.  \notag
\end{align}
 Letting $0<\varepsilon\leq \rho,$ we have
\begin{equation*}
\int_{B(0;\rho)}|{U}_{\varepsilon}|^qdx
=\int_{B(0;\rho)}\frac{1}{(\varepsilon^2+|x|^2)^{\frac{q(N-2s)}{2}}}dx
\geq\int_{B(0;\rho)}\frac{1}{(2\rho^2)^{\frac{q(N-2s)}{2}}}dx=C_4,
\end{equation*}
 for some $C_4=C_4(N,s,\rho).$ Combining this with  \eqref{5.24}, for
$\varepsilon=\big(\lambda^{\frac{2}{2-q}}+\mu^{\frac{2}{2-q}}\big)^{\frac{1}{N-2s}}<\rho$,
$$
\displaystyle\sup_{t\geq t_0}\mathcal {I}_{\lambda,\mu}(tz_0)
\leq\displaystyle\frac{2s}{N}\Big(\frac{k_s\S_{s,\alpha,\beta}}{2}\Big)^{\frac{N}{2s}}
+\O\left(\lambda^{\frac{2}{2-q}}+\mu^{\frac{2}{2-q}}\right)-\frac{t^q_0}{q}(\lambda+\mu)C_4.$$Choosing
$\delta_2>0$ small enough, for all $(\lambda,\mu)\in {\mathscr C}_{\delta_2}$, we have
 $$
\O\left(\lambda^{\frac{2}{2-q}}+\mu^{\frac{2}{2-q}}\right)-\frac{t^q_0}{q}(\lambda+\mu)C_4<-K_0\left(\lambda^{\frac{2}{2-q}}+
 \mu^{\frac{2}{2-q}}\right).$$
If we set $\Lambda^*=\min\{\delta_1,\rho^{N-2s},\delta_2\}>0$,
then for $(\lambda,\mu)\in {\mathscr C}_{\Lambda^*}$,
\be
\label{5.26}\sup_{t\geq0}\mathcal {I}_{\lambda,\mu}(tz_0)<c_{\infty}.
\ee
\par\noindent Finally, we prove that $\alpha^{-}_{\lambda,\mu}<c_{\infty}$ for
all $(\lambda,\mu)\in {\mathscr C}_{\Lambda^*}.$
Recall that
$$z_0=(w_{0,1},w_{0,2})=(\sqrt{\alpha}\phi\mathcal{W}_{\varepsilon},\sqrt{\beta}\phi\mathcal{W}_{\varepsilon}).$$
By Lemma~\ref{Lemma4.6} there is $t_0>0$ such that
$t_0z_0\in\mathcal{N}^{-}_{\lambda,\mu}.$ By the definition of
$\alpha^{-}_{\lambda,\mu}$ and \eqref{5.26}, we conclude 
$$
\alpha^{-}_{\lambda,\mu}\leq\mathcal {I}_{\lambda,\mu}(t_0z_0)\leq
\sup_{t\geq0}\mathcal {I}_{\lambda,\mu}(tz_0)<c_{\infty},
$$
for all $(\lambda,\mu)\in {\mathscr C}_{\Lambda^*}$.
\end{proof}

\noindent
Let $\Lambda^*$ be as in Lemma~\ref{Lemma5.5}.
We prove the existence a local minimizer for $\mathcal
{I}_{\lambda,\mu}$ on $\mathcal{N}^{-}_{\lambda,\mu}$.

\begin{proposition}\label{Proposition5.6}
Let $\Lambda^*>0$ be as in Lemma~\ref{Lemma5.5} and set
$$
\Lambda_2:=\min\{\Lambda^*,(q/2)^{\frac{2}{2-q}}\Lambda_1\}.
$$
For $(\lambda,\mu)\in {\mathscr C}_{\Lambda_2}$,
$\mathcal {I}_{\lambda,\mu}$ has a minimizer $z^-$ in
$\mathcal{N}^{-}_{\lambda,\mu}$ with $\mathcal
{I}_{\lambda,\mu}(z^-)=\alpha^{-}_{\lambda,\mu}$.
Furthermore,  $z^-$ is a positive solution of \eqref{2.6}.
\end{proposition}
\begin{proof}
By (ii) of Proposition \ref{Proposition5.3}, there is a
$(PS)_{\alpha^{-}_{\lambda,\mu}}$ sequence
$\{z_n\}\subset\mathcal{N}^{-}_{\lambda,\mu}$ for $\mathcal
{I}_{\lambda,\mu}$ for all
$$
(\lambda,\mu)\in {\mathscr C}_{(q/2)^{2/(2-q)}\Lambda_1}.
$$
By Lemmas \ref{Lemma3.3} and \ref{Lemma5.5} and (ii) of Theorem
\ref{Theorem4.4}, for $(\lambda,\mu)\in {\mathscr C}_{\Lambda^*}$,
$\mathcal{I}_{\lambda,\mu}$ satisfies the PS condition at the energy level
$\alpha^{-}_{\lambda,\mu}>0.$ Therefore, there exist a
subsequence still denoted by
$\{z_n\}=\{(w_{1,n},w_{2,n})\}\subset\mathcal{N}^{-}_{\lambda,\mu}$
and $z^-:=(w^-_1,w^-_2)\in E$ such that $z_n\rightarrow z^-$
strongly in $E$ and
\be\label{6.10}\mathcal{I}_{\lambda,\mu}(z^-)=\alpha^{-}_{\lambda,\mu}>0~~\mbox{for
all}~ (\lambda,\mu)\in {\mathscr C}_{\Lambda_2}.\ee Now we prove
that $z^-\in\mathcal {N}_{\lambda,\mu}^-.$ By virtue of
$z_n\in\mathcal{N}^{-}_{\lambda,\mu}$  we see that
 $$\langle\mathcal
{R}'_{\lambda,\mu}(z_n),z_n\rangle=(2-q)\|z_n\|^2-2(2^*_s-q)\int_{\Omega}|w_{1,n}|^{\alpha}|w_{2,n}|^{\beta}dx<0,~~\forall
n\in\mathbb{N}.$$Taking the limit in the last inequality and using $z_n\rightarrow z^-$ in $E$, we get 
\be 
\label{6.11}\langle\mathcal
{R}'_{\lambda,\mu}(z^-),z^-\rangle=(2-q)\|z^-\|^2-2(2^*_s-q)\int_{\Omega}|w_{1}^-|^{\alpha}|w_{2}^-|^{\beta}dx\leq0.
\ee
Observe that we must have the strict inequality in \eqref{6.11}.
Otherwise, by \eqref{6.10}, we have $z^-\neq(0,0)$, and so
$z^-\in\mathcal {N}_{\lambda,\mu}^0$, but this contradicts to Lemma
\ref{Lemma4.3}. Thus,  $z^-\in\mathcal {N}_{\lambda,\mu}^-$.  Since
$\mathcal {I}_{\lambda,\mu}(z^-)=\mathcal {I}_{\lambda,\mu}(|z^-|)$
with $|z^-|=(|w^-_1|,|w^-_2|)$, and $|z^-|\in\mathcal
{N}_{\lambda,\mu}^-,$ by Lemma \ref{Lemma4.2} we may assume that
$z^-$ is a nontrivial nonnegative solution of \eqref{2.6}. Moreover,
by $z^-\in\mathcal {N}_{\lambda,\mu}^-$, we get from \eqref{6.11}
that
$$\int_{\Omega}|w^-_1|^{\alpha}|w^-_2|^{\beta}dx>\frac{2-q}{2(2^*_s-q)}\|z^-\|^2>0.$$This
implies that $w^-_1\not\equiv0,w^-_2\not\equiv0.$
 Using the maximum principle as in the end of the proof of
Proposition~\ref{Proposition5.4}, we have $w^-_1,w_2^->0$ in
$\mathcal{C}_{\Omega}.$  Hence, $z^-\in\mathcal {N}_{\lambda,\mu}^-$
is a positive solution for \eqref{2.6}.
\end{proof}

\section{Proof of Theorem~\ref{th11} concluded}
By Proposition~\ref{Proposition5.4},
for $(\lambda,\mu)\in {\mathscr C}_{\Lambda_1}$,
system \eqref{2.6} has a positive solution
$z^+\in\mathcal{N}^{+}_{\lambda,\mu}$.\ By
Proposition \ref{Proposition5.6}, a positive solution
$z^-\in\mathcal{N}^{-}_{\lambda,\mu}$ exists for $(\lambda,\mu)\in {\mathscr C}_{\Lambda_2}.$ Since
$\mathcal{N}^{+}_{\lambda,\mu}\cap\mathcal{N}^{-}_{\lambda,\mu}=\emptyset$,
then  $z^\pm$ are distinct solutions of \eqref{2.6}, so that
$(u^\pm(x),v^\pm(x))=(w^\pm_1(x,0),w^\pm_2(x,0))$ are distinct positive
solutions of \eqref{zwm=1}.\ \qed


\bigskip
\bigskip

\par\noindent {\bf Acknowledgments.}~~The authors would like to
express their sincere gratitude to the anonymous referee for her/his
careful reading the manuscript and valuable comments and
suggestions.

\end{document}